\documentclass{amsart}
\usepackage{paralist}
\usepackage{url}
\usepackage{mathrsfs}
\usepackage{color}
\usepackage{graphicx}
\usepackage{caption}
\usepackage{subcaption}
\usepackage{multirow}

\newcommand{\RR}{{\mathbb R}}

\newcommand{\re}{\mathbb{R}}

\newcommand{\N}{\mathbb{N}}

\newcommand{\eps}{\epsilon}

\def\af{\alpha}

\def\rank{\mbox{rank}}

\newcommand{\sig}{\sigma}
\newcommand{\Sig}{\Sigma}

\newcommand{\reff}[1]{(\ref{#1})}
\newcommand{\pt}{\partial}
\newcommand{\prm}{\prime}
\newcommand{\mc}[1]{\mathcal{#1}}

\newcommand{\bdes}{\begin{description}}
\newcommand{\edes}{\end{description}}
\newcommand{\bal}{\begin{align}}
\newcommand{\eal}{\end{align}}
\newcommand{\bnum}{\begin{enumerate}}
\newcommand{\enum}{\end{enumerate}}
\newcommand{\bit}{\begin{itemize}}
\newcommand{\eit}{\end{itemize}}
\newcommand{\bea}{\begin{eqnarray}}
\newcommand{\eea}{\end{eqnarray}}
\newcommand{\be}{\begin{equation}}
\newcommand{\ee}{\end{equation}}
\newcommand{\baray}{\begin{array}}
\newcommand{\earay}{\end{array}}
\newcommand{\bsry}{\begin{subarray}}
\newcommand{\esry}{\end{subarray}}
\newcommand{\bca}{\begin{cases}}
\newcommand{\eca}{\end{cases}}
\newcommand{\bcen}{\begin{center}}
\newcommand{\ecen}{\end{center}}
\newcommand{\bbm}{\begin{bmatrix}}
\newcommand{\ebm}{\end{bmatrix}}
\newcommand{\bmx}{\begin{matrix}}
\newcommand{\emx}{\end{matrix}}
\newcommand{\bpm}{\begin{pmatrix}}
\newcommand{\epm}{\end{pmatrix}}
\newcommand{\btab}{\begin{tabular}}
\newcommand{\etab}{\end{tabular}}

\newtheorem{theorem}{Theorem}[section]

\newtheorem{prop}[theorem]{Proposition}

\newtheorem{lemma}[theorem]{Lemma}

\theoremstyle{definition}

\newtheorem{exm}[theorem]{Example}
\newtheorem{alg}[theorem]{Algorithm}

\newtheorem{remark}[theorem]{Remark}

\numberwithin{equation}{section}

\begin{document}

\title[Bilevel Polynomial Programs]
{Bilevel Polynomial Programs and Semidefinite Relaxation Methods}
%\author{Jiawang Nie,  Li Wang and Jane J. Ye}
\author{Jiawang Nie}
\address{Department of Mathematics,
	University of California, 9500 Gilman Drive, La Jolla, CA, USA, 92093.}
\email{njw@math.ucsd.edu}
\author{Li Wang}
\address{Department of Mathematics, Statistics, and Computer Science, University of Illinois at Chicago, Chicago, IL, USA, 60607.}
\email{liwang8@uic.edu}
\author{Jane J. Ye}
\address{Department of Mathematics and Statistics, University of Victoria, Victoria, B.C., Canada, V8W 2Y2.}
\email{janeye@uvic.ca}



\begin{abstract}
A bilevel program is an optimization problem whose constraints
involve the solution set to another
optimization problem parameterized {by} upper level variables.
This paper studies bilevel polynomial programs (BPPs), i.e.,
all the functions are polynomials.
We reformulate BPPs equivalently as semi-infinite polynomial programs (SIPPs),
using Fritz John conditions and Jacobian representations.
Combining the exchange technique and Lasserre type semidefinite relaxations,
we propose a numerical method for solving bilevel polynomial programs.
For simple BPPs, we prove the convergence
to global optimal solutions.  Numerical experiments
are presented to show the efficiency of the proposed algorithm.
\end{abstract}

\maketitle

\section{Introduction}

We consider the {\it bilevel polynomial program} (BPP):
\be \label{bilevel:pp}
(P): \left\{
\begin{aligned}
F^* := \min\limits_{x\in \re^n,y\in \re^p}&\ F(x,y) \\
\text{s.t.} \quad &\ G_i(x,y)\geq 0, \, i=1,\cdots,m_1, \\
& \  y\in S(x),
%:=\text{arg}\min\limits_{z\in Z(x) }  f(x,z),
\end{aligned}
\right.
\ee
where $F$ and all $G_i$ are real polynomials
in $(x,y)$, and $S(x)$ is the set of global minimizers
of the following lower level program,
which is parameterized by $x$,
\be \label{df:P(x)}
 \quad
\min\limits_{ z\in \re^p} \quad f(x,z)  \quad
\text{s.t.} \quad g_j(x,z) \geq 0,\, j=1,\cdots, m_2.
\ee
In \reff{df:P(x)}, $f$ and each $g_j$ are polynomials in $(x,z)$.
For convenience, denote
\[
{
Z(x) :=\{ z \in \re^p \mid  g_j(x,z) \geq 0,\, j=1,\cdots, m_2 \},
}
\]
the feasible set of \reff{df:P(x)}.
The inequalities $G_i(x,y) \geq 0$ are called upper (or outer) level constraints,
while $g_j(x,z)\geq 0$ are called lower (or inner) level constraints.
When $m_1=0$ (resp., $m_2=0$),
there are no upper (resp., lower) level constraints.
%Accordingly, $(P)$ is called the upper level optimization \textcolor{red}{problem},
%while \reff{df:P(x)} is called the \textcolor{red}{lower level optimization problem}.
%Sometimes, \reff{bilevel:pp} is called the outer optimization \textcolor{red}{problem}
%and \reff{df:P(x)} is called the inner optimization \textcolor{red}{problem}.
Similarly, $F(x,y)$ is the upper level (or outer) objective,
and $f(x,z)$ is the lower level (or inner) objective.
Denote the set
%of all $(x,y)$ that are feasible
%for both upper and \textcolor{red}{lower level optimization problems}
\be \label{set:U}
\mathcal{U}: =\left\{(x,y)
\left|\baray{c}
G_i(x,y)\geq 0 \, (i=1,\cdots, m_1), \, \\
g_j(x,y) \geq 0\, (j=1,\cdots, m_2)
\earay \right.
\right\}.
\ee
Then the feasible set of $(P)$ is the intersection
\be \label{set:feasible:P}
\mathcal{U}\cap \{(x,y): y\in S(x)\}.
\ee
Throughout the paper, we assume that for all $(x,y)\in \mathcal{U}$,
$S(x)\neq \emptyset$ and consequently the feasible set of $(P)$ is nonempty.
When the lower level feasible set $Z(x)\equiv Z$ is independent of $x$,
we call the problem $(P)$ a {\it simple bilevel polynomial program} (SBPP).
The SBPP is not {mathematically simple
but actually quite challenging}.
SBPPs have important applications in economics,
e.g., the moral hazard model of the principal-agent problem \cite{mirrlees1999theory}.
When {the feasible set of the lower level program} $Z(x)$ depends on $x$, the problem $(P)$ is called a
{\it general bilevel polynomial program} (GBPP). GBPP
is also an effective modelling tool for many applications in various fields;
see e.g. \cite{Dempebook,Dempebook2} and the references {therein}.

\subsection{Background}

The bilevel program is a class of difficult optimization problems.
Even for the case where all the functions are linear,
the problem is NP-hard  \cite{ben1990computational}.
A general approach for solving bilevel programs is to
transform them into single level optimization problems.
A commonly used technique is to replace the {lower level  program}
by its Kurash-Kuhn-Tucker (KKT) conditions.
When the lower level program {involves} inequality constraints,
the reduced problem becomes a so-called
  {\it mathematical program with equilibrium constraints} (MPEC) \cite{luo1996mathematical,opac-b1102982}.
If the {lower level  program} is nonconvex,
the  optimal solution of a bilevel program may not even be a stationary point of
the reduced single level optimization {problem} by using the KKT conditions.
This was shown by a counter example due to Mirrlees~\cite{mirrlees1999theory}.
Moreover, even if the lower level  {program} is convex,
{it was shown in
\cite{Dempe2012MP} that a local solution to
the MPEC obtained by replacing the lower level program
by its KKT conditions may not be a local solution to
the original bilevel program.
Recently, \cite{allende2013solving} proposed to replace the lower level  program with its Fritz John conditions instead of its KKT conditions. However, it was shown in \cite{Dempe2016} that {the} same difficulties remain, i.e., solutions to the MPEC obtained by replacing the lower level program by its Fritz John conditions may
{not be} the solutions to the original bilevel program.}

An alternative approach for solving BPPs is to use the value function
\cite{outrata1990numerical,Jane1995},
which gives an equivalent reformulation.
%
%However, the reformulated problem is nonsmooth and nonconvex,
%and it never satisfies the generalized
%Mangasarian-Fromovitz constraint qualification.
%
However, the optimal solution of the bilevel
{program} may not be a stationary point of the value function reformulation.
To overcome this difficulty, {\cite{YeSIAM2010} proposed to combine
the  KKT and the value function reformulations.}
%\cite{Colson2005Bilevel,JaneYe2006,Jane1995,YeSIAM2010,SBLYe2013}.
%However all the algorithms can only be used to find stationary points.
Over the past two decades, many numerical algorithms were proposed
for solving bilevel programs. However, most of them assume that
the {lower level  program} is convex, with few exceptions {\cite{SBLYe2013,Mitsos2006Thesis,
mitsos2008global,outrata1990numerical,xu2013smoothing,xu2014smoothing,XuYeZhang}.
In \cite{Mitsos2006Thesis,mitsos2008global},
an algorithm using the branch and bound in combination
with the exchange technique was proposed to find approximate global optimal solutions.
Recently, the smoothing techniques were used to
find stationary points of the valued function or
the combined reformulation of simple bilevel programs
\cite{SBLYe2013,xu2013smoothing,xu2014smoothing,XuYeZhang}. }

%
%Recently, the smoothing techniques were used to find stationary points
%of simple bilevel programs
%\cite{SBLYe2013,xu2013smoothing,xu2014smoothing,XuYeZhang}.
%

In general, it is quite difficult to find global {minimizers} of nonconvex optimization problems. However, when the functions are polynomials, there exists much work
on computing global optimizers,
by using Lasserre type semidefinite relaxations \cite{Las01}.
We refer to \cite{Las09,ML2014} for the recent work in this area.
Recently, Jeyakumar, Lasserre, Li and Pham \cite{Lasbilevel2015}
worked on simple bilevel polynomial programs.
When the lower level  program \reff{df:P(x)} is convex for each fixed $x$,
they transformed \reff{bilevel:pp} into a single level polynomial program,
by using Fritz John conditions and the multipliers
to replace the lower level  program, and
globally solving it by using Lasserre type  relaxations.
When \reff{df:P(x)} is nonconvex for some $x$,
by approximating the value function of lower level programs
by a sequence of polynomials, they propose to reformulate
\reff{bilevel:pp} with approximate lower level programs
by the value function approach,
and globally solving the resulting sequence of polynomial programs
by using Lasserre type relaxations.
The work \cite{Lasbilevel2015} is very inspiring,
because polynomial optimization techniques were proposed to solve BPPs.
In this paper, we also use Lasserre type semidefinite relaxations
to solve BPPs, but we make different reformulations,
by using Jacobian representations
and the exchange technique in semi-infinite programming.

\subsection{From BPP to SIPP}

A bilevel {program}  can be reformulated as a
{\it semi-infinite program} (SIP). Thus, the classical methods
(e.g., the exchange method {\cite{Infinitely1976,Hettich1993,WangSIPP2014}}) for SIPs can be applied to solve bilevel programs.
For convenience of {introduction}, at the moment, we consider SBPPs,
i.e., the feasible set $Z(x)\equiv Z$ in \reff{df:P(x)} is independent of $x$.

Before reformulating BPPs as SIPs, we show the fact:
\be \label{df:H(xyz)}
y\in S(x)  \Longleftrightarrow y\in Z, \quad\, H(x,y,z)
\geq 0 \,\, (\forall~z\in Z),
\ee
where $H(x,y,z) := f(x,z) - f(x,y)$.
Clearly, the ``$\Rightarrow$" direction is true.
Let us prove {the reverse direction}. Let $v(x)$ denote the value function:
\be \label{def:v(x)}
v(x) \,:= \,  \inf_{z\in Z }  f(x,z).
\ee
If $(x,y)$ satisfies the right hand side conditions in \reff{df:H(xyz)},
then
\[
\inf\limits_{z\in Z} \, H(x,y,z)  =v(x) -f(x,y)\geq 0.
\]
%is feasible for $(\widetilde{P})$, then $\displaystyle v(x)=\min_{z\in Z(x)} f(x,z)$ and
Since $y\in Z$, we have
$v(x)-f(x,y)\leq 0$. Combining these two inequalities, we get
\[
v(x)={\inf_{z\in Z}} f(x,z) = f(x,y)
\]
and hence $ y\in S(x).$
%Thus,
%\be \label{condition:min:G:0}
%y\in S(x) \qquad \Longleftrightarrow \qquad
%y\in Z, \quad
%\min\limits_{z\in Z} H(x,y,z) = 0.
%\ee

By the fact \reff{df:H(xyz)}, the problem $(P)$ is equivalent to
\be \label{bilevel:sipp:reform}
(\widetilde{P}): \left\{
\baray{rl}
{F}^* := \min\limits_{x \in \re^n, \, y\in {Z}} &  F(x,y) \\
\text{s.t.}& \ G_i(x,y)\geq 0,\ i=1,\dots, m_1, \\
& \ H(x,y,z)\geq 0,~\forall~ z\in Z.
\earay
\right.
\ee
The problem $(\widetilde{P})$ is a semi-infinite polynomial program (SIPP),
if the set $Z$ is infinite.
%We refer to \cite{WangSIPP2014} for SIPPs.
%\textcolor{red}{Let us prove the equivalence between problems $(P)$ and $(\widetilde{P})$}.
%A classical method for solving SIPPs is the so-called exchange method
%\textcolor{red}{(\cite{Infinitely1976,Hettich1993,WangSIPP2014})}.
Hence, the exchange method can be used to solve $(\widetilde{P})$.
Suppose $Z_k$ is a finite grid of $Z$.
Replacing $Z$ by $Z_k$ in $(\widetilde{P})$, we get:
\be \label{pop:tilde(Pk)}
(\widetilde{P}_k): \left\{
\begin{aligned}
{F}^*_k := \min\limits_{x \in \re^n, \, y\in {Z} }&\ F(x,y) \\
\text{s.t.} &\ G_i(x,y)\geq 0,\ i=1,\dots, m_1, \\
&\ H(x,y,z)\geq 0, ~\forall~ z\in Z_k .
\end{aligned}
\right.
\ee
{The feasible set of $(\widetilde{P}_k)$
contains that of $(\widetilde{P})$}. Hence,
\[
{F}^*_k \leq {F}^*.
\]
Since $Z_k$ is a finite set, $(\widetilde{P}_k)$
is a polynomial optimization problem. If, for some $Z_k$,
we can get an optimizer $(x^k,y^k)$ of $(\widetilde{P}_k)$ such that
\be \label{v(xk)-f(xkyk)}
v(x^k) - f(x^k,y^k) \geq 0,
\ee
then $y^k\in S(x^k)$ and $(x^k,y^k)$ is feasible for $(\widetilde{P})$.
In such case, $(x^k,y^k)$ must be a global optimizer of $(\widetilde{P})$.
Otherwise, if \reff{v(xk)-f(xkyk)} fails to hold,
then there exists $z^{k}\in Z$ such that
\[
f(x^k, z^k) - f(x^k, y^k) < 0.
\]
For such a case, we can construct the new grid set as
\[
Z_{k+1} := Z_k\cup \{z^{k}\},
\]
and then solve the new problem $(\widetilde{P}_{k+1})$
with the grid set $Z_{k+1}$. Repeating this process,
we can get an algorithm for solving $(\widetilde{P})$ approximately.

How does the above approach work in computational practice?
Does it converge to global optimizers?
Each subproblem $(\widetilde{P}_k)$ is a polynomial optimization problem,
which is generally nonconvex.
Theoretically, it is NP-hard to solve polynomial optimization globally.
However, in practice, it can be solved successfully
by Lasserre type semidefinite relaxations (cf.~\cite{Las01,Las09}).
Recently, it was shown in \cite{NiefiniteLas} that
Lasserre type semidefinite relaxations are generally
tight for solving polynomial optimization problems.
About the convergence,
we can see that $\{ F_k^* \}$ is a sequence of monotonically
increasing lower bounds for the global optimal value $F^*$, i.e.,
\[
F_1^* \leq  \cdots \leq F_k^*  \leq {F}^*_{k+1} \leq \cdots \leq {F}^*.
\]
By a standard analysis for SIP (cf.~\cite{Hettich1993}),
one can expect the convergence $F_k^*  \to F^*$, under some conditions.
However, we would like to point out that
the above exchange process typically converges
{\it very slowly} for solving BPPs. A major reason is that
the feasible set of $(\widetilde{P}_k)$ is {\it much {larger}} than
that of $(\widetilde{P})$. Indeed, the dimension of the feasible set of
$(\widetilde{P}_k)$ is typically larger than that of $(\widetilde{P})$.
This is because, for every feasible $(x,y)$ in $(\widetilde{P})$,
$y$ must also satisfy optimality conditions for the lower level
 {program} \reff{df:P(x)}. In the meanwhile, the $y$ in $(\widetilde{P}_k)$
does not satisfy such optimality conditions.
Typically, for $(\widetilde{P}_k)$ to approximate $(\widetilde{P})$
reasonably well, the grid set $Z_k$ should be {\it very big}.
In practice, the above standard exchange method is not efficient
for solving BPPs.

\subsection{Contributions}

In this paper, we propose an efficient computational method for solving BPPs.
First, we transform a BPP into an equivalent SIPP,
by using Fritz John conditions and Jacobian representations.
Then, we propose a new algorithm for solving BPPs,
by using the exchange technique and Lasserre type
semidefinite relaxations.

For each $(x,y)$ that is feasible for \reff{bilevel:pp},
$y$ is a minimizer for the lower level program \reff{df:P(x)}
parameterized by $x$.
If some constraint qualification conditions are satisfied,
the KKT conditions hold. If such qualification conditions fail to hold,
the KKT conditions might not be satisfied. However,
the Fritz John conditions always hold  for \reff{df:P(x)}
(cf.~\cite[\S3.3.5]{Bertsekas1990}
{and
 \cite{bental1976JOTA} for optimality conditions
for convex programs without constraint qualifications).
}
%\textcolor{blue}{Note that even for convex programs, the Fritz John condition may not characterized the set of global optimal solution unless a Slater point exists (see e.g. \cite{bental1976JOTA} ).}
So, we can add the Fritz John conditions to $(\widetilde{P})$,
while the problem is not changed.
%%%%%%%%%%%%%%%%%%%%%%%%%%%%%
%%%%%%%%%%%%%%%%%%%%%%%%%%%%%
\iffalse

However since $y^k$ obtained would have to be at least a Fritz John point of the lower level problem $P(x^k)$, applying the exchange method to $(\widetilde{P})$ with  the Fritz John condition added would speed up the convergence significantly.
Alternatively if a certain constraint qualification for the lower level problem {$\mathcal{L}(x)$} holds, then one may also consider adding the KKT condition to speed up the convergence. The advantage of using the Fritz John condition instead of the KKT condition is that no constraint qualifications are required for the lower level problem.

\fi
%%%%%%%%%%%%%%%%%%%%%%%
%%%%%%%%%%%%%%%%%%%%%%%
A disadvantage of using Fritz John conditions
is the usage of multipliers, which need to be considered as new variables.
Typically, using multipliers will make the
polynomial program much harder to solve, because of
new additional variables.
To overcome this difficulty, the technique in \cite[\S2]{Nie2011Jacobian}
can be applied to avoid the usage of multipliers.
This technique is known as Jacobian representations
for optimality conditions.
%
% By using the Jacobian representation of the set of generalized critical points which are the solutions of the relaxed Fritz John condition, we can speed up the convergence significantly and in the mean time the number of variables and equations would not have increased.
%

% To accelerate the convergence of the exchange method for solving problem $(\widetilde{P})$, we consider to add some redundant  constraints that restrict $y$ to be a \emph{generalized critical point} of lower level program, which  includes all feasible points of the lower level program satisfying the Fritz John condition. Moreover
%Instead of using the Fritz John condition by adding new multiplier variables, we consider to represent the set of the generalized critical points by using   a technique  called the Jacobian representation  proposed by Nie in \cite[Section 2]{Nie2011Jacobian}.
 %The advantage of this reformulation is that it guarantees to include all global minimizers $S(x)$ of lower level program. No matter $y\in S(x)$ is degenerate or not.

The above observations motivate us to solve
bilevel polynomial programs, by combining Fritz John conditions,
Jacobian representations, Lasserre relaxations,
and the exchange technique.
Our major results are as follows:

\begin{itemize}

%\item The bilevel programming problem we consider is very general. Except requiring all the functions are polynomials, we do not require many assumptions commonly used in the bilevel programming literature. For example, we do not assume that for each $x$, the lower level problem has a global optimal solution. In stead we only assume that the BPP has a global solution.

\item Unlike some prior methods for solving BPPs,
we do not assume the KKT conditions hold for the {lower level  program} \reff{df:P(x)}.
Instead, we use the Fritz John conditions.
This is because the KKT conditions may fail to hold for the {lower level  program} \reff{df:P(x)},
while the Fritz John conditions always hold.
By using Jacobian representations, the usage of multipliers can be avoided.
This greatly improves the computational efficiency.
%
% to refine the feasible region of the approximate problem. Since the Fritz John condition holds at any optimal solution, our algorithm even applies to the case when the global solutions of the lower level program  do not satisfy the KKT condition.
%

\item For simple bilevel polynomial programs,
we propose an algorithm using
Jacobian representations, Lasserre relaxations
and the exchange technique.
Its convergence to global minimizers is proved.
The numerical experiments show that it
is efficient for solving SBPPs.

\item For general bilevel polynomial programs,
we can apply the same algorithm, using
Jacobian representations, Lasserre relaxations
and the exchange technique. The numerical experiments show that it
works well for some GBPPs, while it is not theoretically guaranteed to
get global optimizers. However,
its convergence to global optimality
can be proved under some assumptions.

\end{itemize}

The paper is organized as follows: In Section~\ref{sc:prlm}, we review
some preliminaries in polynomial optimization and Jacobian representations.
In Section~\ref{sc:sbpp}, we propose a method for solving
simple bilevel polynomial programs and prove its convergence.
In Section~\ref{sec:GBPP}, we consider general bilevel polynomial programs
and show how the algorithm works.
In Section~\ref{sc:numexp}, we present numerical experiments to demonstrate
the efficiency of the proposed methods.
{
In Section~\ref{sc:condis}, we make some conclusions
and discussions about our method.}

\section{Preliminaries}
\label{sc:prlm}

\noindent {\bf Notation.}
The symbol $\N$ (resp., $\re$ ,$\mathbb{C}$) denotes the set of nonnegative integers (resp., real numbers, complex numbers). For an integer $n>0$, $[n]$ denotes the set $\{1,\cdots,n\}$.
%
%For $x \in \re^n$, $x_i$ denotes the $i$-th component of $x$.
%
For $x := (x_1, \ldots, x_n) $ and $\af := (\af_1, \ldots, \af_n)$,
denote the monomial
\[
  x^\af \, := \, x_1^{\af_1}\cdots x_n^{\af_n}.
\]
For a finite set $T$, $|T|$ denotes its cardinality.
The symbol $\mathbb{R}[x] := \mathbb{R}[x_1,\cdots,x_n]$
denotes the ring of polynomials in
$x:=(x_1,\cdots,x_n)$ with real coefficients whereas
$\mathbb{R}[x]_k$ denotes its subspace of polynomials of degree at most $k$.
For a polynomial $p\in \RR[x]$,  {define the set product
\[
 p\cdot \RR[x] : =\{pq \mid  q\in \RR[x]\}.
\]
It is the principal ideal generated by $p$}. For a symmetric matrix
$W$, $W\succeq 0$ (resp., $\succ 0$) means that $W$ is positive
semidefinite (resp., definite). For a vector $u\in \re^n$, $\| u \|$
denotes the standard Euclidean norm.
The gradient of a function $f(x)$ is denoted as $\nabla f(x)$.
If $f(x,z)$ is a function in both $x$ and $z$,
then $\nabla_z f(x,z)$ denotes the gradient with respect to $z$.
For an optimization problem, {\tt argmin} denotes
the set of its optimizers.

\subsection{Polynomial optimization}\label{Poly}

An ideal $I$ in $\re[x]$ is a subset of $\re[x]$
such that $ I \cdot \re[x] \subseteq I$
and $I+I \subseteq I$. For a tuple $p=(p_1,\ldots,p_r)$ in $\re[x]$,
$I(p)$ denotes the smallest ideal containing all $p_i$, i.e.,
\[
I(p) = p_1 \cdot \re[x] + \cdots + p_r \cdot \re[x].
\]
The $k$th {\it truncation} of the ideal $I(p)$,
denoted as $I_k(p)$, is the set
\[
p_1 \cdot \re[x]_{k-\deg(p_1)} + \cdots + {p_r}  \cdot \re[x]_{k-\deg(p_r)}.
\]
For the polynomial tuple $p$, denote its real zero set
\begin{align*}
\mc{V}(p)  := \{v \in \re^n \mid \,  p(v) = 0 \}.
\end{align*}

A polynomial $\sig \in \re[x]$ is said to be a sum of squares (SOS)
if $\sig = a_1^2+\cdots+ a_k^2$ for some $a_1,\ldots, a_k \in \re[x]$.
The set of all SOS polynomials in $x$ is denoted as $\Sig[x]$.
For a degree $m$, denote the truncation
\[
\Sig[x]_m := \Sig[x] \cap \re[x]_m.
\]
For a tuple $q=(q_1,\ldots,q_t)$,
its {\it quadratic module} is the set
\[
Q(q):=  \Sig[x] + q_1 \cdot \Sig[x] + \cdots + q_t \cdot \Sig[x].
\]
The $k$-th truncation of $Q(q)$ is the set
\[
\Sig[x]_{2k} + q_1 \cdot \Sig[x]_{d_1} + \cdots + q_t \cdot \Sig[x]_{d_t}
\]
where each $d_i = 2k - \deg(q_i)$. For the tuple $q$,
denote the basic semialgebraic set
\[
\mc{S}(q) := \{ v \in \re^n \mid  q(v) \geq 0 \}.
\]

For the polynomial tuples $p$ and $q$ as above,
if $f \in I(p) + {Q(q)}$, then clearly $f \geq 0$ on the set
$\mc{V}(p) \cap \mc{S}(q)$. However, the reverse is not necessarily true.
The sum $I(p) + Q(q)$ is said to be {\it archimedean}
if there exists $b \in I(p) + Q(q)$ such that
$S(b) = \{ v \in \re^n: b(v) \geq 0\}$ is {a} compact set in $\re^n$. {Putinar \cite{Putinar1993} proved that
if a polynomial $f > 0$ on $\mc{V}(p) \cap \mc{S}(q)$ and if $I(p) + Q(q)$ is archimedean},
then $f  \in I(p) + Q(q)$. When $f$ is only nonnegative
(but not strictly positive) on $\mc{V}(p) \cap \mc{S}(q)$, we still have
$f  \in I(p) + Q(q)$, under some general conditions
(cf.~\cite{NiefiniteLas}).

Now, we review Lasserre type semidefinite relaxations in polynomial optimization.
More details can be found in \cite{Las01,Las09,ML2014}.
Consider the general polynomial optimization problem:
\be \label{pro::opti}
\left\{\baray{rl}
f_{\min}:= \quad \underset{x\in\RR^n}{\min} &  f(x)  \\
\text{s.t.} &  p(x) = 0, \, q(x) \geq 0,
\earay \right.
\ee
where $f \in \re[x]$ and $p,q$ are tuples of polynomials.
The feasible set of \reff{pro::opti}
is precisely the intersection $\mc{V}(p) \cap \mc{S}(q)$.
The Lasserre's hierarchy of semidefinite relaxations
for solving \reff{pro::opti} is ($k=1,2,\ldots$):
\be \label{k:SOS:las}
\left\{\baray{rl}
f_k := \, \max  & \gamma  \\
\text{s.t.}  &  f-\gamma \in I_{2k}(p) + Q_k(q).
\earay \right.
\ee
When the set $I(p) + Q(q)$ is archimedean, Lasserre proved the convergence
\[
f_k \quad \to \quad f_{\min},\quad \text{as }k\rightarrow \infty.
\]
{If there exist $k<\infty$ such that $f_k =f_{\min}$,
the Lasserre's hierarchy is said to have finite convergence.
{Under the archimedeanness and some standard  conditions
in optimization known to be generic (i.e., linear independence constraint qualification, strict complementarity and second order sufficiency conditions)},
the Lasserre's hierarchy has finite convergence.
This was recently shown in \cite{NiefiniteLas}.
{On the other hand, there exist special polynomial optimization problems
for which the Lasserre's hierarchy fails to have finite convergence.
But, such special problems belong to a set of measure zero
in the space of input polynomials, as shown in \cite{NiefiniteLas}.}
%{\bf I suggest to remove the following colored words?}
%\textcolor{red}{If there exist $k<\infty$ such that $f_k =f_{\min}$, we say Lasserre type semidefinite relaxations have finite convergence.} \textcolor{blue}{The finite convergence of Lasserre type semidefinite relaxation is not guarateed for every polynomial program \cite[Chapter 5]{Las09}. However under the standard archimedean condition, the finite convergence  happens generically, i.e.,  under the linear independence constraint qualification, strict complementarity and the second order sufficiency conditions at every glocal minimizers  (cf.~\cite{NiefiniteLas}).}
Moreover, we can also get global minimizers of
\reff{pro::opti} by using the flat extension or flat truncation condition
(cf.~\cite{Nie2011certify}).
The optimization {problem} \reff{k:SOS:las} can be solved as
a semidefinite program, so it can be solved by semidefinite program packages
(e.g., SeDuMi \cite{sedumi}, SDPT3 \cite{sdpt3}).
A convenient and efficient software for using
Lasserre relaxations is {\tt GloptiPoly~3} \cite{Gloptipoly}.

\subsection{Jacobian representations}

We consider the polynomial optimization problem
that is similar to the {lower level  program} \reff{df:P(x)}:
\be \label{pro::opti:2}
 \underset{z\in\RR^p}{\min} \quad f(z) \quad
\text{s.t.} \quad g_1(z)\ge 0, \ldots, g_m(z) \geq 0,
\ee
where $f, g_1, \ldots, g_m \in\RR[z]:=\re[z_1,\ldots,z_p]$.
Let $Z$ be the feasible set of \reff{pro::opti:2}.
For $z\in Z$, let $J(z)$ denote the index set of
active constraining functions at $z$.

%
%Generally, we can make the following assumption.
%\begin{ass}\label{assumption:1}
%For every $z\in Z$, at most $p$ of $g_1,\ldots, g_m$ vanish at $z$.
%\end{ass}
%
%For a subset $J = \{i_1,\dots, i_k\}\subseteq [m]$, let
%\begin{eqnarray*}
% g_J \, := \, \big( g_{i_1}, \dots, g_{i_k} \big).
%\end{eqnarray*}
%Denote by $\nabla g_J(z)$ the Jacobian of the tuple $g_J$ at $z$.
%

Suppose $z^*$ is an optimizer of \reff{pro::opti:2}.
By the Fritz John condition (cf.~\cite[\S3.3.5]{Bertsekas1990}), there exists
{$ (\mu_0,\mu_1, \ldots, \mu_m) \ne 0 $} such that
\be \label{FJ:cond}
\mu_0\nabla f(z^*) -\sum\limits_{i=1}^m\mu_i \nabla g_i(z^*) =0,
\quad \mu_ig_i(z^*)=0 \ (i\in [m]).
\ee
{A point like $z^*$ satisfying \reff{FJ:cond}  is called a Fritz John point.}
If we only consider active constraints, the above is then reduced to
\be \label{act:FJ:cd}
\mu_0\nabla f(z^*) -\sum\limits_{i\in J(z^*)}\mu_i \nabla g_i(z^*) =0.
\ee
The condition~\reff{FJ:cond} uses multipliers $\mu_0, \ldots, \mu_m$,
which are often not known in advance. If we consider them as new variables,
then it would increase the number of variables significantly.
For the index set $J = \{ i_1, \ldots, i_k\}$, denote the matrix
\[
B[J,z]:= \bbm \nabla f(z) & \nabla g_{i_1} (z) & \cdots
& \nabla g_{i_k} (z) \ebm.
\]
Then condition~\reff{act:FJ:cd} means that the matrix
$B[J(z^*), z^*]$ is rank deficient, i.e.,
\[
\rank \, B[J(z^*), z^*] \, \leq \, | J(z^*)|.
\]
The matrix $B[J(z^*), z^*]$ depends on the active set $J(z^*)$,
which is typically unknown in advance.

The technique in \cite[\S2]{Nie2011Jacobian} can be applied to
get explicit equations for Fritz John points,
without using multipliers $\mu_i$.
 For a subset $J=\{i_1, \ldots, i_k\} \subseteq [m]$
with cardinality $|J| \leq \min\{m,p-1\}$,
write its complement as $J^c := [m]\backslash J.$ Then
\[
B[J,z] \mbox{ is rank defincient } \Longleftrightarrow \mbox{ all $(k+1) \times (k+1)$ minors of $B[J,z]$ are zeros}.
\]
%
%\textcolor{blue}{\bf I am not sure whether ``zero''
%or ``zeros'' are correct. ``zeros'' is used to define solutions to an equation.}
%
There are totally $\binom{p}{k+1}$ equations defined by {such} minors.
However, this number can be significantly reduced by
% many of these equations defined by the minors are redundant and hence
%List all the $(k+1) \times (k+1)$ minors of the matrix $B[J,z]$ in sequence as
%\be \label{jacob:monior}
%\eta^J_1, \ldots, \textcolor{red}{\eta^J_{\ell(J)}}.
%\ee
using the method in \cite[\S2]{Nie2011Jacobian}.
The number of equations, for characterizing that $ B[J,z] \mbox{ is rank defincient}$,
can be reduced to
\[
\ell(J) \, := \, p(k+1)-(k+1)^2+1.
\]
It is much smaller than $ p \choose k+1$.
For cleanness of the paper, we do not repeat the construction of
these minimum number defining polynomials.
Interested readers {are referred to
\cite[\S2]{Nie2011Jacobian} for the details.
List all the defining polynomials,
which make $ B[J,z]$ {rank deficient}, as
\be \label{jacob:monior}
\eta^J_1, \ldots,  \eta^J_{\ell(J)}.
\ee\noindent}Consider the products of
these polynomials with {$g_j$'s}:
	\be \label{eta:dot:g}
	\eta^J_1 \cdot \Big( \underset{ j \in J^c}{\Pi}  g_j  \Big), \quad \ldots, \quad
	\eta^J_{\ell(J)} \cdot \Big( \underset{ j \in J^c}{\Pi} g_j \big).
\ee
They are all polynomials in $z$. The active set $J(z)$
is undetermined, unless $z$ is known.
We consider all possible polynomials as in (\ref{eta:dot:g}),
for all $J \subseteq [m]$, and collect them together.
For convenience of notation, denote all such polynomials as
\be \label{jac:poly:psi}
	\psi_1, \quad \ldots, \quad \psi_{L},
\ee
where the number
\begin{align*}
L \,\,&= \, \sum\limits_{J \subseteq [m],|J| \leq \min\{m,p-1\} }
\ell(J) \\
& =  \sum\limits_{ 0 \leq k \leq \min\{m,p-1\} }
\binom{m}{k} \big( p(k+1) - (k+1)^2 +1 \big).
\end{align*}
When $m,k$ are big, the number $L$ would be very large.
This is an unfavorable feature of Jacobian representations.

%	
%umber of minors of the matrix $B[J,z]$. Instead of using all ,
%\cite[\S2]{Nie2011Jacobian} proposed a method to use the minimal number
%of $(k+1) \times (k+1)$ minors to define the
%	
% using minimal number of equations to define the Fritz John points. }

We point out that the Fritz John points can be characterized
by using the polynomials $\psi_1, \ldots, \psi_L$.
Define the set of all Fritz John points:
\be \label{KFJ:pts}
K_{FJ} := \left\{z\in \re^p  \left|
\baray{c}
 \exists (\mu_0,\mu_1, \ldots, \mu_m) \ne 0, \,
%\textcolor{blue}{g_i(z)\geq 0,}\
 \mu_ig_i(z) = 0 \,(i\in [m]),  \\
 \mu_0 \nabla f(z)  -  \sum\limits_{i=1}^m\mu_i \nabla g_i(z)=0.
\earay
\right. \right\}.
\ee
Let $W$ be the set of real zeros of polynomials $\psi_j(z)$, i.e.,
\be \label{len2}
W  \, =  \, \{z\in \re^p \mid \psi_1(z) = \cdots = \psi_L(z)=0\}.
\ee
{It is interesting to note that the sets $K_{FJ}$ and $W$ are equal.}

% The set $W$ can be defined by a smaller number of equations.
%By the constructions of polynomials $\psi_j$, one can easily show that
%\be \label{KFJ:subset:W}
%K_{FJ} \subseteq W.
%\ee
%This is because, for every Fritz John point $z^*$,
%the matrix $B[J(z^*), z^*]$ has rank $\leq |J(z^*)|$,
%so the polynomials $\psi_j$ must vanish at $z^*$.
%

\begin{lemma}\label{lemma:gcp}
For $K_{FJ}, W$ as in \reff{KFJ:pts}-\reff{len2},
it holds that $K_{FJ} = W$.
\end{lemma}
\begin{proof} First, we prove that $W\subseteq K_{FJ}$.
Choose an arbitrary $u\in W$, and let $J(u)$ be the active set at $u$.
If $|J(u)| \geq p$, then the gradients $\nabla f(u)$
and $\nabla g_j(u)$ ($j\in J(u)$) must be linearly dependent,
so $u \in K_{FJ}$. Next, we suppose $|J(u)| < p$.
Note that $g_j(u)>0$ for all $j\in J(u)^c$.
By the construction, some of $\psi_1, \ldots, \psi_L$
are the polynomials as in \reff{eta:dot:g}
\[
{\eta_t^{J(u)}} \cdot \Big( \underset{ j \in J(u)^c}{\Pi}  g_j  \Big).
\]
Thus, $\psi(u) =0$ implies that all the polynomials {$\eta_t^{J(u)}$}
vanish at $u$. By their definition, we know the matrix
$B[J(u),u]$ does not have full column rank.
This means that $u\in K_{FJ}$.

%
%{\bf Case }	$I(u)<p$ { By (\ref{len}), since $g_j(u)=0 \  \forall j\in I(u)^c$, we have $$ \eta_{i}^{I(u)} (B^{I(u)}(u))=\psi_i^{I(u)}(u) = 0\quad \forall i=1,\cdots,len(I(u)).$$
%  By virtue of (\ref{lennew}), matrix $B^{I(u)}(u)$ does not have full column rank. So there exists nonzero $(\mu_0,\mu_{I(u)})\in \RR^{|I(u)|+1}$ such that
%  $$0=\mu_0 \nabla f(u)-\sum_{i\in I(u)} \mu_i \nabla g_i(u).$$ Setting  $\mu_j = 0$ for all  $j\in I(u)^c$, we have $u\in K_{z}$.}
%	
%{\bf Case }  $I(u) = p$ { There are two cases depending on whether the columns of the matrix $\nabla g_{I(u)}(u)$ are linearly dependent or not. If the columns of $\nabla g_{I(u)}$ are linearly dependent, then there exists nonzero $\mu_{I(u)}\in \RR^{I(u)}$ such that
%  $$0=\sum_{i\in I(u)} \mu_i \nabla g_i(u).$$ Setting  $\mu_0=0, \mu_{I(u)^{c}} = 0$, we have $u\in K_{z}$. Otherwise suppose that matrix $[\nabla g_{I(u)}]$ is nonsingular.  If $\nabla f(u) \neq 0$, there exists $\mu_0=1,\mu = [\nabla g_{I(u)}]^{-1} \nabla f(u) $, such that $\mu_0 \nabla f(u) =  \sum_{i\in I(u)}\mu_i \nabla g_i(u) $, i.e., $u\in K_{z}$. If $\nabla f(u) = 0$, we have $\mu_0 = 1,\mu = 0$ such that $u\in K_z$.  }
%
%

Second, we show that $K_{FJ} \subseteq W$.
Choose an arbitrary $u\in K_{FJ}$.
\bit

\item Case I: $J(u)=\emptyset$.
Then $\nabla f(u)=0$. The first column of matrices {$B[\emptyset,u]$} is zero,
so all $\eta_t^\emptyset$ and $\psi_j$
vanishes at $u$ and hence $u\in W$.

\item Case II: $J(u) \not =\emptyset$. Let $I\subseteq [m]$
be an arbitrary index set with $|I|\leq \min\{m,p-1\}$.
If $J(u)\not \subseteq I$, then at least one $j\in I^c$ belongs to $J(u)$.
Thus, at least one $j \in I^c$ satisfies $g_j(u)=0$, so all the polynomials
\[
{\eta_t^{I}} \cdot \Big( \underset{ j \in I^c}{\Pi} g_j  \Big)
\]
vanish at $u$.
If $J(u) \subseteq I$, then $\mu_i g_i(u)=0$ implies that
$\mu_i=0$ for all $i\in I^c$.
By definition of $K_{FJ}$, the matrix {$B[I,u]$}
does not have full column rank.
So, the minors $\eta_i^I$ of {$B[I,u]$} vanish at $u$.
By the construction of $\psi_i$, we know all $\psi_i$ vanish at $u$,
so $ u\in W$.
\eit
The proof is completed by combining the above two cases.
\end{proof}

\section{Simple bilevel polynomial programs}
\label{sc:sbpp}

In this section, we study simple bilevel polynomial programs (SBPPs)
and give an algorithm for computing global optimizers.
For SBPPs as in \reff{bilevel:pp}, the feasible set $Z(x)$
for the {lower level  program} \reff{df:P(x)} is independent of $x$.
Assume that $Z(x)$ is constantly the semialgebraic set
\be \label{Z:sbpp}
Z:=\{z\in \re^p \mid g_1(z) \geq 0, \ldots, g_{m_2}(z) \geq 0\},
\ee
for given polynomials $g_1, \ldots, g_{m_2}$ in $z:=(z_1,\ldots, z_p)$.
For each pair $(x,y)$ that is feasible in \reff{bilevel:pp},
$y$ is an optimizer for \reff{df:P(x)} which now becomes
\be \label{inopt:minfg}
\min\limits_{ z\in \re^p} \quad f(x,z)  \quad
\text{s.t.} \quad g_1(z) \geq 0, \ldots, g_{m_2}(z) \geq 0.
\ee
Note that the inner objective $f$ still depends on $x$.
So, $y$ must be a Fritz John point of \reff{inopt:minfg}, i.e.,
there exists {$ (\mu_0,\mu_1, \ldots, \mu_{m_2}) \ne 0 $} satisfying
\[
\mu_0 \nabla_z f(x,y) - \sum_{j \in [m_2] } \mu_j \nabla_z g_j(y) = 0,
\quad \mu_j g_j(y) = 0 \, (j \in [m_2]).
\]
Let $K_{FJ}(x)$ denote the set of all Fritz John points of \reff{inopt:minfg}.
The set $K_{FJ}(x)$ can be characterized by Jacobian representations.
Let $\psi_1, \ldots, \psi_L$ be the polynomials constructed as in
\reff{jac:poly:psi}. Note that each $\psi_j$
is now a polynomial in $(x,z)$, because the objective of
\reff{inopt:minfg} depends on $x$. Thus, each $(x,y)$
feasible for \reff{bilevel:pp} satisfies
\[
\psi_1(x,y) = \cdots = \psi_L(x,y) = 0.
\]
For convenience of notation, denote the polynomial tuples
\be \label{sc3:xi:psi}
\xi :=\big( G_1, \ldots, G_{m_1}, g_1, \ldots, g_{m_2} \big), \quad
\psi := \big( \psi_1, \ldots, \psi_L \big),
\ee
{We call $\psi(x,y)=0$} a {\it Jacobian equation}.
Then, the SBPP as in \reff{bilevel:pp} is equivalent to
the following SIPP:
\be \label{sipp:FsGH}
\left\{
\baray{rl} F^* := \min\limits_{x \in \re^n, y \in \re^p}  & \, F(x,y) \\
\text{s.t.} & \, \psi(x,y) = 0, \, \xi(x,y) \geq 0, \\
& \, H(x,y,z) \geq 0, \, \forall~z\in Z.
\earay \right.
\ee
In the above, $H(x,y,z)$ is defined as in \reff{df:H(xyz)}.

\subsection{A semidefinite algorithm for SBPP}

We have seen that the SBPP \reff{bilevel:pp} is equivalent to
\reff{sipp:FsGH}, which is an SIPP. So, we can apply
the {exchange} method to solve it.
The basic idea of ``exchange" is that
we replace $Z$ by a finite grid set $Z_k$ in \reff{sipp:FsGH},
and then solve it for a global minimizer $(x^k, y^k)$ by Lasserre relaxations.
If $(x^k, y^k)$ is feasible for \reff{bilevel:pp}, we stop;
otherwise, we compute global minimizers of $H(x^k,y^k, z)$ and
add them to $Z_k$. Repeat this process until the convergence condition is met.
We call $(x^*,y^*)$ a global minimizer of \reff{bilevel:pp},
up to a tolerance parameter $\eps>0$, if $(x^*,y^*)$
is a global minimizer of the following approximate SIPP:
\be \label{pop:tilde(Pepsilon)}
 \left\{
\baray{rl}
{F}^*_\epsilon := \min\limits_{x \in \re^n ,y \in \re^p} &  F(x,y) \\
\text{s.t.} \quad & \, \psi(x,y) = 0, \, \xi(x,y) \geq 0, \\
& \ H(x,y,z)\geq -\epsilon, ~\forall~ z\in Z.
\earay
\right.
\ee
Summarizing the above, we get the following algorithm.

\begin{alg}\label{alg:exchange:method}
(A Semidefinite Relaxation Algorithm for SBPP.)

\medskip
\noindent\textbf{Input:} Polynomials $F$, $f$,
$G_1,\ldots, G_{m_1}$, $g_1,\ldots, g_{m_2}$
for the SBPP as in \reff{bilevel:pp},
a tolerance parameter $\eps \geq 0$, and a maximum number $k_{\max}$
{of iterations}.

\medskip
\noindent\textbf{Output:} The set $\mc{X}^*$ of
global minimizers of \reff{bilevel:pp}, up to the tolerance $\eps$.

\medskip
\begin{enumerate}[\itshape Step 1]
%
%\item Formulate set $\mathcal{X}$ in \reff{def:jacob:set} by the Jacobian
%representation.
%

\item Let $Z_0=\emptyset$, $\mc{X}^*= \emptyset$ and $k=0$.

\item
\label{the Jacobian representation}
Apply Lasserre  relaxations to solve
\be \label{alg:simple:sub1}
(P_k): \left\{
\baray{rl}
 F_k^* :=\min\limits_{x \in \re^n, y \in \re^p} &\, F(x,y)\\
\text{s.t.} &\, \psi(x,y) = 0, \, \xi(x,y) \geq 0, \\
&\, H(x,y,z)\geq 0 \,  (\forall~z\in Z_k),
\earay
\right.
\ee
and get the set
$S_k = \{(x^k_1,y_1^k),\cdots,(x^k_{r_k},y^k_{r_k})\}$
of its global minimizers.

\item \label{step3}
For each $i=1,\cdots,r_k$, do the following:
\begin{enumerate}[\upshape (a)] \label{item::step3}

\item Apply Lasserre relaxations to solve
\be  \label{alg:simple:sub2}
(Q^k_i):\quad \left\{
\baray{rl}
 v^k_i := \min\limits_{z \in \re^p} &\  H(x^k_i,y^k_i,z)\\
 \text{s.t.}&\, \psi(x_i^k, z) = 0, \\
 & \, g_1(z) \geq 0, \ldots, g_{m_2}(z) \geq 0,
\earay
\right.
\ee
and get the set $T^k_i=\left\{z^k_{i,j} :\ j=1,\cdots,t^k_i\right\}$
of its global minimizers.

\item
If $v^k_i \geq - \eps$,
then update
$\mathcal{X}^* : =\mathcal{X}^* \cup \{(x^k_i,y_i^k)\}$.
\end{enumerate}

\item
If $\mathcal{X}^*\neq \emptyset$ or $k>k_{\max}$, stop;
otherwise, {update $Z_{k}$ to $Z_{k+1}$ as}
\be \label{Z:k+1:sbpp}
Z_{k+1} :=Z_{k} \cup T^k_1 \cup \cdots \cup T^k_{r_k}.
\ee
Let $k:=k+1$ and go to Step
\ref{the Jacobian representation}.
\end{enumerate}
\end{alg}

For the exchange method  to solve the SIPP \reff{sipp:FsGH} successfully,
the two subproblems~\reff{alg:simple:sub1} and \reff{alg:simple:sub2}
need to be solved globally in each iteration.
This can be done by Lasserre's hierarchy of
semidefinite relaxations (cf.~\S\ref{Poly}).
\bit
\item [A)]
For solving \reff{alg:simple:sub1} by Lasserre's hierarchy, we get a sequence of
monotonically increasing lower bounds for $F_k^*$, say,
$\{ \rho_\ell \}_{\ell=1}^\infty$, that is,
\[
\rho_1 \leq \cdots \leq \rho_\ell \leq \cdots  \leq F_k^*.
\]
Here, $\ell$ is a relaxation order.
If for some value of $\ell$
we get a feasible point $(\hat{x}, \hat{y})$ for \reff{alg:simple:sub1}
such that $F(\hat{x}, \hat{y}) = \rho_\ell$, then we must have
\be \label{eq:F(xy)=Fk}
F(\hat{x}, \hat{y}) = F_k^* = \rho_\ell,
\ee
and know $(\hat{x}, \hat{y})$ is a global minimizer.
This certifies that the Lasserre's relaxation of order $\ell$ is exact
and \reff{alg:simple:sub1} is solved globally,
i.e., Lasserre's hierarchy has finite convergence.
As recently shown in \cite{NiefiniteLas},
Lasserre's hierarchy has finite convergence,
when the archimedeanness and {some standard  
conditions well-known in optimization to be generic
(i.e., linear independence constraint qualification,
strict complementarity and second order sufficiency conditions}) hold.

\item [B)]{
For a given polynomial optimization problem, there exist a sufficient (and almost necessary)
condition for detecting whether or not Lasserre's hierarchy has finite convergence.
The condition is {\it flat truncation}, proposed in \cite{Nie2011certify}.
It was proved in \cite{Nie2011certify} that
Lasserre's hierarchy has finite convergence
if the flat truncation condition
is satisfied.
When the flat truncation condition holds,
we can also get the point
$(\hat{x}, \hat{y})$ in \reff{eq:F(xy)=Fk}.
In all of our numerical examples, the flat truncation condition
is satisfied, so we know that Lasserre relaxations
solved them exactly.}
{There exist special optimization problems
for which Lasserre relaxations are not exact (see e.g. \cite[Chapter 5]{Las09}).
Even for the worst case that
Lasserre's hierarchy fails to have finite convergence,
flat truncation is still the right condition for checking asymptotic convergence.}
This is proved in \cite[\S3]{Nie2011certify}.

\item [C)]
In computational practice, semidefinite programs cannot be solved
exactly, because round-off errors always exist in computers.
Therefore, if $F(\hat{x}, \hat{y}) \approx \rho_\ell$,
it is reasonable to claim that
\reff{alg:simple:sub1} is solved globally.
This numerical issue is a common feature of most computational methods.

\item [D)]
For the same reasons as above, the subproblem~\reff{alg:simple:sub2}
can also be solved globally by Lasserre's relaxations.
Moreover, \reff{alg:simple:sub2} uses the equation
$\psi(x_i^k,z)=0$, obtained from Jacobian representation.
As shown in \cite{Nie2011Jacobian}, Lasserre's hierarchy of relaxations,
in combination with Jacobian representations, always has finite convergence,
under some nonsingularity conditions.
This result has been improved in \cite[Theorem 3.9]{GuoWang2014}
under weaker conditions.
Flat truncation can be used to detect
the convergence (cf.~\cite[\S4.2]{Nie2011certify}).

\item [E)]
{
For all $\epsilon_{1}> \epsilon_2 >0$, it is easy to see that
$F_{\epsilon_1}^*\leq F_{\epsilon_2}^*\leq F^*$ and hence the feasible region and the optimal value of the bilevel problems are monotone. Indeed, we can prove
$\lim\limits_{\epsilon\rightarrow 0^+}~F_{\epsilon}^*=F^* $
and the continuity of the optimal solutions;
see \cite[Theorem 4.1]{SBLYe2013} for the result and a detailed proof. 
However, we should point out that if $\epsilon >0$ is not small enough, then the solution of the approximate bilevel program may be very different from the one for the original bilevel program.
We refer to \cite[Example 4.1]{Mitsos2006Thesis}.
}

\item [F)] In Step~3 of Algorithm~\ref{alg:exchange:method},
the value of $v_i^k$ is a measure for the feasibility of
$(x_i^k, y_i^k)$ in \reff{sipp:FsGH}.
This is because $(x_i^k, y_i^k)$ is a feasible point for \reff{sipp:FsGH}
if and only if $v_i^k \geq 0$.
By using the exchange method,
the subproblem \reff{alg:simple:sub1} is only an
approximation for \reff{sipp:FsGH},
so typically we have $v_i^k < 0$ {if $(x_i^k, y_i^k)$ is infeasible for \reff{sipp:FsGH}}. The closer $v_i^k$ is to zero, the better \reff{alg:simple:sub1} approximates \reff{sipp:FsGH}.
\eit

\subsection{Two features of the algorithm}
As in the introduction, we do not apply the exchange method directly
to \reff{bilevel:sipp:reform}, but instead to \reff{sipp:FsGH}.
Both \reff{bilevel:sipp:reform} and \reff{sipp:FsGH}
are SIPPs that are equivalent to the SBPP \reff{bilevel:pp}.
As the numerical experiments will show, the SIPP \reff{sipp:FsGH}
is much easier to {solve} by the exchange method.
This is because, {the Jacobian equation} $\psi(x,y)=0$ in \reff{sipp:FsGH}
makes it much easier for \reff{alg:simple:sub1}
to approximate \reff{sipp:FsGH} accurately.
Typically, for a finite grid set $Z_k$ of $Z$,
the feasible sets of \reff{sipp:FsGH} and \reff{alg:simple:sub1}
have the same dimension. However, the feasible set of
\reff{bilevel:sipp:reform} has smaller dimension than that of \reff{pop:tilde(Pk)}.
Thus, it is usually very difficult for \reff{pop:tilde(Pk)} to approximate
\reff{bilevel:sipp:reform} accurately, by choosing a finite set $Z_k$.
In contrast, it is often much easier for \reff{alg:simple:sub1} to approximate
\reff{sipp:FsGH} accurately. We illustrate this fact
by the following example.

\begin{exm} \label{exm:appendix:9new}
(\cite[Example 3.19]{BIEXM})
Consider the SBPP:
\be  \label{exm:comparison}
\left\{
\baray{rl}
\min\limits_{x\in \RR,y\in \RR} &~~ F(x,y): = xy-y+\frac{1}{2}y^2\\
\text{s.t.} &  1-x^2 \geq 0, \quad 1 - y^2 \geq 0, \\
&  ~~ %\|x\|^2+\|y\|^2 = 1,\\
y\in S(x) :=  \underset{1-z^2 \geq 0}{\tt argmin} \quad
f(x,z): = -xz^2+\frac{1}{2}z^4.
\earay
\right.
\ee
Since $f(x,z) =  \frac{1}{2}(z^2-x)^2-\frac{1}{2}x^2$, one can see that
\[
S(x)=\left\{
\begin{aligned}
& 0,\quad \quad x\in [-1,0),\\
& \pm \sqrt{x}, \quad x\in [0,1].
\end{aligned}
\right.
\]
Therefore, the outer objective $F(x,y)$ can be expressed as
\[
F(x,y)=\left\{
\begin{aligned}
& 0,\quad \quad \quad \quad\quad \quad \quad x\in [-1,0),\\
& \frac{1}{2}x\pm (x-1)\sqrt{x}, \quad x\in [0,1].
\end{aligned}
\right.
\]
So, the optimal solution and the optimal value of \reff{exm:comparison}
are ($a = \frac{\sqrt{13}-1}{6}$):
\[
(x^*,y^*) = (a^2,a) \approx (0.1886, \, 0.4343),
\quad F^* = \frac{1}{2}a^2+a^3-a  \approx  -0.2581.
\]
If Algorithm \ref{alg:exchange:method}
is applied without using the Jacobian {equation} $\psi(x,y)=0$,
the computational results are shown in Table \ref{Table:comparision:1}.
%
%It takes more than four steps for the algorithm  to converge.
% From step 5, software Gloptipoly \cite{Gloptipoly}
%has numerical troubles in extracting the global minimizers
%of subproblem $(P_k)$ and the algorithm can not continue.
%
The problem \reff{exm:comparison} cannot be solved reasonably well. In the contrast,  if we apply Algorithm \ref{alg:exchange:method}
with the Jacobian {equation} $\psi(x,y)=0$, then \reff{exm:comparison}
is solved very well. The computational results are shown in
Table~\ref{Table:comparision:2}.
It takes only two iterations for the algorithm to converge.
\begin{table}[htb]
\caption{Computational results without $\psi(x,y) =0$}
\label{Table:comparision:1}
	\centering
	\begin{scriptsize}
		\begin{tabular}{|c|c|c|c|c|c|c|c|} \hline
			{\tt Iter} $k$   & $(x_i^k,y_i^k)$  & $z_{i,j}^k$ & $F_k^*$ &  $v_i^k$   \\ 	\hline
			0  & (-1,1)  & 4.098e-13 & -1.5000 & -1.5000 \\
			\hline
			1  &  (0.1505, 0.5486)  &  $\pm 0.3879$ &  -0.3156  &  -0.0113  \\
			\hline
			2  & (0.0752, 0.3879) & $\pm 0.2743$ & -0.2835& -0.0028\\
			\hline
			3  & (0.2088, 0.5179) & $\pm 0.4569 $ & -0.2754 & -0.0018\\
			\hline
            4  & cannot be solved & ... & ... & ... \\ \hline
		\end{tabular}
	\end{scriptsize}
\end{table}
\begin{table}[htb]
\caption{Computational results with $\psi(x,y) =0$ }
\label{Table:comparision:2}
	\centering
	\begin{scriptsize}
		\begin{tabular}{|c|c|c|c|c|c|c|c|} \hline
			{\tt Iter} $k$   & $(x_i^k,y_i^k)$  & {$z_{i,j}^k$} & $F_k^*$ &  $v_i^k$   \\ 	\hline
			0  & (-1,1)  & 3.283e-21 & -1.5000 & -1.5000 \\
			\hline
			1  &  (0.1886,0.4342)  &  $\pm 0.4342$ &  -0.2581  &  -3.625e-12 \\ \hline
		\end{tabular}
	\end{scriptsize}
\end{table}
\end{exm}
%
%In the following example,  it is interesting to note that
%although the feasible region of the problem is convex,
%the constraint functions are not concave and hence
%$P(x)$ is not a standard convex program. Since the KKT condition does not for the lower level optimal solution, the first order approach by which the lower level problem
%is replaced by its KKT condition does not work and the combined approach
%using both the value function and the KKT condition does not work either.
%We now describe in detail how to find the global optimal solution
%by using the algorithm we proposed by using this example.
%

For the lower level program \reff{df:P(x)},
the KKT conditions may fail to hold.
In such a case, the classical methods
which replace \reff{df:P(x)} by the KKT conditions,
do not work at all. However, such problems can also be
solved efficiently by Algorithm~\ref{alg:exchange:method}.
The following are two such examples.

%{\bf Any specific reference?
%If so, some examples are required to show the advantage of our method.}

\begin{exm} (\cite[Example 2.4]{Dempe2012MP}) Consider the following SBPP:
	\begin{equation}\label{exm:comparsion:1}
	F^*:=\min\limits_{x\in \RR,y\in \RR} (x-1)^2+y^2\quad \text{s.t.}\quad
y\in S(x) := \underset{ z\in Z := \{z\in \RR| z^2\leq 0\}}{\tt argmin}~x^2z.
	\end{equation}
It is easy to see that the global minimizer of this problem is
$(x^*,y^*) = (1,0)$. The set $Z=\{0\}$ is convex.
By using the multiplier variable $\lambda$,
we get a single level optimization problem:
	\begin{equation*} \left\{
		\begin{aligned}
			r^* :=	\min\limits_{x\in \RR,y\in \RR,\lambda\in \RR}&\ (x-1)^2+y^2\\
			\text{s.t.}&\ x^2+2\lambda y =0, \, \lambda \geq 0,\, y^2\leq 0,\, \lambda y^2=0.\\
		\end{aligned}
		\right.
	\end{equation*}
The feasible points of this problem are $(0,0,\lambda)$
with $\lambda\geq 0$. We have $r^* = 1>F^*$.
The KKT reformulation approach fails in this example,
since $y^*\in S(x^*)$ is not a KKT point.
We solve the SBPP problem \reff{exm:comparsion:1}
by Algorithm~\ref{alg:exchange:method}. The Jacobian equation is
$\psi(x,y) = x^2y^2=0$, and we reformulate the problem as:
	\begin{equation*} \left\{
		\begin{aligned}
			s^* :=	\min\limits_{x\in \RR,y\in \RR}&\ (x-1)^2+y^2\\
			\text{s.t.}&\ x^2(z-y)\geq 0,\quad \forall z\in Z,\\
			&\ \psi(x,y) = x^2y^2=0.
		\end{aligned}
		\right.
	\end{equation*}
This problem is not an SIPP actually, since the set $Z$
only has one feasible point. At the initial step,
we find its optimal solution $(x^*,y^*) = (1,0)$,
and it is easy to check that $\min\limits_{z\in Z} H(x^*,y^*,z) = 0$,
which certifies that it is the global minimizer of
the SBPP problem \reff{exm:comparsion:1}.
\end{exm}

\begin{exm}   \label{exm:degenerate}
Consider the SBPP:
\be  \label{exm:degenerate:P}
\left\{
\baray{rl}
\min\limits_{x\in \re, \, y \in \re^2}  &~~ F(x,y): = x+y_1+y_2\\
\text{s.t.} & ~~ x-2 \geq 0, \quad 3-x \geq 0, \\
& ~~ y\in S(x) :=  \underset{z\in Z}{\tt argmin} \quad
f(x,z): = x(z_1+z_2), \\
\earay
\right.
\ee
where set $Z$ is defined by the inequalities:
\[
g_1(z):=z^2_1-z_2^2 - (z^2_1+z^2_2)^2 \geq 0,
\quad g_2(z) := z_1\geq 0.
\]
%\begin{figure}
%\includegraphics[width=0.5\textwidth]{Figure1.pdf}
%\caption{Feasible region $Z$ of lower level problem.}\label{exm:degenerate:Z}
%\end{figure}
%%%%%%%%%%%%%%%%%%%%%%%%%%%%%
\iffalse

\begin{figure}
	\centering
	\begin{subfigure}{0.4\textwidth}
		\includegraphics[width=1.0\textwidth]{Figure1.pdf}
		\caption{}
		\label{fig:exm:denegerate:1}
	\end{subfigure}%
	\quad \quad \quad
	\begin{subfigure}{0.4\textwidth}
		\includegraphics[width=1.0\textwidth]{Figure2.pdf}
		\caption{}
		\label{fig:exm:denegerate:2}
	\end{subfigure}
	\caption{\small{(A) Feasible region $Z$ of the lower level program in (\ref{exm:degenerate:P}). (B)  Feasible region of the lower level program with the Jacobian representation.}}
	\label{fig:exm:degenerate:all}
\end{figure}

The set $Z$ is shown in the shaded area in Figure 1(A).
\fi
%%%%%%%%%%%%%%%%%%%%%%%%%%%%%%%%
For all $x\in [2,3]$, one can check that $S(x) = \{(0,0)\}$.
Clearly, the global minimizer of \reff{exm:degenerate:P} is
$(x^*,y^*) = (2,0,0)$, and the optimal value $F^*= 2$. At $z^*= (0,0) $,
\[
\nabla_{z} f(x,z^*) = \bbm x \\ x\ebm,
\nabla_z g_1(z^*) = \bbm 0 \\ 0\ebm,
\nabla_z g_2(z^*) = \bbm 1 \\ 0\ebm.
\]
The KKT condition does not hold for the {lower level  program},
since $\nabla_{z} f(x,z^*)$ is not a linear combination
of $\nabla_z g_1(z^*)$ and $\nabla_z g_2(z^*)$.
%
% there does not exist $(\lambda_1,\lambda_2)\in \RR^2_+$ such that
%$$\nabla_{z} f(x,z^*) = \lambda_1 \nabla g_1(z^*) +\lambda_2 \nabla g_2(z^*).$$
%
By \cite[Proposition 3.4]{NiefiniteLas}, Lasserre relaxations in \reff{k:SOS:las}
do not have finite convergence for solving the {lower level  program}.
%
%The set of all generalized critical points for the lower level problem $P(x)$ is
%\begin{equation*}
%K(x)=\left\{z\in Z \Bigg|
%\begin{aligned}
%& \text{ there exists nonzero} \ \  (\mu_0, \mu)\in \RR^3  \text{ such that } \\
%& \mu_0  \bbm x \\ x\ebm=\mu_1 \nabla g_1(z)+\mu_2 \bbm 1 \\ 0\ebm, g_1(z)\mu_1=0, z_1\mu_2=0.
%\end{aligned}
%\right\}.
%\end{equation*}
%Obviously taking $\mu_0=\mu_2=0$ and $\mu_1\not =0$, we have that $(0,0)\in K(x)$ is a generalized critical point.
%Since $x\not =0$, it is not possible to have a generalized critical point with both $g_1(z)\not =0, z_1\not =0$. Now consider the case when $z_1\not =0$ but $g_1(z)=0$. In this case since $x\not =0$, by scaling the generalized  critical point must satisfy
%$$ \bbm 1 \\ 1\ebm=\lambda \nabla g_1(z)$$
%for some $\lambda$. Therefore the generalized critical point must satisfy
%\begin{eqnarray*}
%&& g_1(z)=0\\
%&&  z_1+z_2+2(z_2-z_1)(z_1^2+z_2^2)=0.
%\end{eqnarray*}
%Solving the above system, we obtain $(z_1,z_2) \approx (0.8990,0.2409)$ as another generalized critical point.
%
One {can check} that
\[
K_{FJ}(x) \, = \, \{(0,0), (0.8990,0.2409)\}
\footnote{They are the solutions of the equations
$g_1(z)=0, \, z_1+z_2+2(z_2-z_1)(z_1^2+z_2^2)=0.$}
,
\]
for all feasible $x$.
By Jacobian representation of $K_{FJ}(x)$, we get
\[
\psi(x,z) = \Big(
xg_1(z)g_2(z), \quad -xz_1( z_1+z_2+2(z_2-z_1)(z_1^2+z_2^2)), \quad
-xg_1(z)
\Big).
\]
Next, we apply Algorithm \ref{alg:exchange:method}
to solve \reff{exm:degenerate:P}.
Indeed, for $k=0$, $Z_0 = \emptyset$, we get
\[ (x_1^0,y_1^0) \approx ( 2.0000,0.0000,0.0000), \]
which is the true global minimizer.
We also get
\[
z_1^0 \approx (4.6320, -4.6330)\times 10^{-5}, \quad
v_1^0 \approx -5.2510\times 10^{-8}.
\]
For a small value of $\eps$ (e.g., $10^{-6}$),
Algorithm \ref{alg:exchange:method}
terminates successfully with the global minimizer of
\reff{exm:degenerate:P}.
\end{exm}

\subsection{Convergence analysis}

We study the convergence properties of Algorithm \ref{alg:exchange:method}.
For theoretical analysis,
one is mostly interested in its performance when
the tolerance parameter $\eps =0$ or
the maximum iteration number $k_{\max}=\infty$.

\begin{theorem} \label{theorem:exchange:simple}
For the simple bilevel polynomial program as in \reff{bilevel:pp},
assume the lower level  program is as in \reff{inopt:minfg}.
Suppose the subproblems $(P_k)$ and each $(Q_i^k)$
are solved globally by Lasserre relaxations.

\bit
\item [(i)]
Assume $\eps =0$.
If Algorithm \ref{alg:exchange:method} stops for some
$ k < k_{\max}$,
then each $(x^*,y^*) \in \mc{X}^*$
is a global minimizer of \reff{bilevel:pp}.

\item [(ii)] Assume $\eps =0$, $k_{\max}=\infty$, and the union
$\cup_{k\geq 0} Z_k$ is bounded.
Suppose Algorithm \ref{alg:exchange:method}
does not stop and each $S_k \ne \emptyset$ is finite.
Let $(x^*,y^*)$ be an arbitrary accumulation point of
the set $\cup_{k\geq 0} S_k$. If the value function $v(x)$,
as in \reff{def:v(x)}, is continuous at $x^*$,
then $(x^*, y^*)$  is a global minimizer of
the SBPP problem \reff{bilevel:pp}.

\item [(iii)]
Assume $k_{\max}=\infty$, the union $\cup_{k\geq 0} Z_k$ is bounded,
the set $\Xi = \{(x,y): \, \psi(x,y) =0, \xi(x,y) \geq 0\}$ is compact.
{Let $\Xi_1 =\{x : \,  \exists y, (x,y)\in \Xi\}$,
which is the projection of $\Xi$ onto the $x$-space.
Suppose $v(x)$ is continuous on $\Xi_1$.}
Then, for all $\eps>0$, Algorithm \ref{alg:exchange:method}
must terminate within finitely many steps,
and each {$(\bar{x}, \bar{y})\in \mc{X}^*$}
is a global minimizer of the approximate SIPP (\ref{pop:tilde(Pepsilon)}).

\eit
\end{theorem}

\begin{proof}	
(i) The SBPP \reff{bilevel:pp} is equivalent to \reff{sipp:FsGH}.
Note that each optimal value $F_k^* \leq F^*$ and the sequence
$\{ F_k^* \}$ is monotonically increasing.
If Algorithm \ref{alg:exchange:method} stops at the $k$-th iteration,
then each $(x^*,y^*) \in {\mathcal{X}^*}$ is feasible for \reff{sipp:FsGH},
and also feasible for \reff{bilevel:pp}, so it holds that
\[
F^* \geq F_{k}^* = F(x^*, y^*) \geq F^*.
\]
This implies that $(x^*,y^*)$ is a global optimizer of problem \reff{bilevel:pp}.

(ii) Suppose Algorithm \ref{alg:exchange:method} does not stop
and each $S_k \ne \emptyset$ is finite.
For each accumulation point $(x^*, y^*)$ of the union $\cup_{k\geq 0} S_k$,
there exists a sequence $\{ k_\ell \}$ of integers such that
$k_\ell \to \infty$ as $\ell \to \infty$ and {
\[
(x^{k_\ell}, y^{k_\ell})  \to  (x^*, y^*), \quad \mbox{where  each } (x^{k_\ell}, y^{k_\ell}) \in S_{k_\ell}.
\]
Since the feasible set of problem $(P_{k_{\ell}})$ contains  the one for problem \reff{bilevel:pp}, we have $ F^*_{k_\ell} = F(x^{k_\ell}, y^{k_\ell})\leq F^* \text{ and } $  hence $F(x^*,y^*)\leq F^*$ by the continuity of $F$.  To show the opposite inequality it suffices to show that $(x^*,y^*)$ is feasible for problem \reff{bilevel:pp}. {Recall that 
the function $\xi$ is defined as in \reff{sc3:xi:psi}.} Since $\xi(x^{k_\ell},y^{k_\ell}) \geq 0$  and $\psi(x^{k_\ell},y^{k_\ell}) =0$, by the continuity  of the mappings $\xi, \psi$, we have $\xi(x^*,y^*)
\geq 0$
 and $\psi(x^{*},y^{*}) =0$.}
Define the function
\begin{equation}\label{def:phi}
\phi(x,y):= {\inf\limits_{z\in Z}} H(x,y,z).
\end{equation}

%
%To show that $(x^*,y^*)$ is feasible for \reff{sipp:FsGH},
%it remains to show that $\phi(x^*,y^*)\geq 0$.
Clearly, $\phi(x,y)  = v(x) - f(x,y)$,
and $\phi(x^*,y^*) =0$ if and only if $(x^*,y^*)$
is a feasible point for \reff{bilevel:pp}.
By the definition of $v(x)$ as in \reff{def:v(x)}
and that $v(x)$ is continuous at $x^*$, we always have
$
\phi(x^*, y^*)\leq 0.
$
To prove $\phi(x^*,y^*) =0$, it remains to show $\phi(x^*, y^*) \geq 0$.
For all $k^{\prm}$ and for all $k_{\ell} \geq k^{\prm}$,
the point $(x^{k_\ell},y^{k_\ell})$ is feasible for the subproblem $(P_{k^\prm})$, so
\[
H(x^{k_\ell},y^{k_\ell},z)\geq 0\quad \forall z\in Z_{k^\prm}.
\]
Letting $\ell \to \infty$, we then get
\be \label{H(x*y*z)>=0}
H(x^*,y^*,z)\geq 0\quad \forall z\in Z_{k^\prm}.
\ee
The above is true for all $k^\prm$. In Algorithm~\ref{alg:exchange:method},
for each $k_{\ell}$, there exists $z^{k_\ell} \in T_i^{k_\ell}$,
for some $i$, such that
\[
{\phi(x^{k_\ell},y^{k_\ell})} = H(x^{k_\ell},y^{k_\ell},z^{k_\ell}).
\]
Since $ z^{k_\ell} \in Z_{k_\ell+1}$, by \reff{H(x*y*z)>=0}, we know
\[
H(x^*,y^*,z^{k_\ell} )\geq 0.
\]
Therefore, it holds that
\begin{equation}
\baray{rcl}
\phi(x^*, y^*) & = & \phi(x^{k_\ell},y^{k_\ell})+\phi(x^*,y^*)
-\phi(x^{k_\ell},y^{k_\ell})\\
			& \geq & [H(x^{k_\ell},y^{k_\ell},z^{k_\ell})
 - H(x^*, y^*,z^{k_\ell})] + \\
 &  & [\phi(x^*, y^*)-\phi(x^{k_\ell},y^{k_\ell})].
\earay
\end{equation}
Since $z^{k_\ell}$ belongs to the bounded set $\cup_{k\geq 0} Z_k$,
there exists a subsequence $z^{k_{\ell,j}}$ such that $z^{k_{\ell,j}} \to z^* \in Z$.
The polynomial $H(x,y,z)$ is continuous at $(x^*, y^*, z^*)$. Since $v(x)$ is continuous at $x^*$,
$\phi(x,y) = v(x) - f(x,y)$ is also continuous at $(x^*, y^*)$.
Letting $\ell \to \infty$, we get $\phi(x^*, y^*)\geq 0$. Thus, $(x^*, y^*)$ is feasible for \reff{sipp:FsGH} and so $F(x^*,y^*) \geq F^*$.
In the earlier, we already proved $F(x^*,y^*) \leq F^*$,
so $(x^*, y^*)$ is a global optimizer of \reff{sipp:FsGH},
i.e., $(x^*, y^*)$  is a global minimizer of
the SBPP problem \reff{bilevel:pp}.

(iii)
Suppose otherwise the algorithm does not stop within finitely many steps.
Then there exist a sequence $\{(x^k, y^k, z^k)\}$ such that
$(x^k,y^k) \in S_k$, $z^k \in \cup_{i=1}^{r_k} T_i^k$,
\[
H(x^k, y^k, z^k) < -\eps
\]
for all $k$. Note that $(x^k,y^k) \in \Xi$ and $z^k \in Z_{k+1}$.
{By the assumption that $\Xi$ is compact and
$\cup_{k\geq 0} Z_k$ is bounded, the sequence $\{(x^k, y^k, z^k)\}$
has a convergent subsequence, say,
\[
(x^{k_\ell}, y^{k_\ell}, z^{k_\ell})  \, \to \, {(x^*,y^*,z^*)} \qquad
\mbox{ as } \quad \ell \to \infty.
\]
So, it holds that  ${(x^*,y^*)}\in \Xi,~z^*\in Z$ and
$H{(x^*,y^*,z^*)}  \leq -\eps$.
Since $\Xi$ is compact, the projection set $\Xi_1$ is also compact,
hence ${x^*} \in \Xi_1$. By the assumption, we know
$v(x)$ is continuous at ${x^*}$.
Similar to the proof in (ii), we have $\phi{(x^*,y^*)} =0$, then ${(x^*,y^*)}$ is a feasible point for \reff{bilevel:pp}, and we will get $$H{(x^*,y^*,z^*)} = f{(x^*,z^*)}-f{(x^*,y^*)}\geq 0.$$  However, this contradicts that
$H{(x^*,y^*,z^*)} \leq -\eps$.
Therefore, Algorithm \ref{alg:exchange:method} must terminate within finitely many steps.
}
%Next, let us show $\phi(x^*,y^*) =0$. We can use the similar proof of (ii). By the continuity of the mappings $\xi, \psi$, we have $\xi(x^*,y^*)
%\geq 0$ and $\psi(x^{*},y^{*}) =0$, i.e., $(x^*,y^*)\in \Xi$. Function $\phi(x,y) = v(x)-f(x,y)$, since $v(x)$ is continuous at $x^*$, then $\phi(x^*,y^*)\leq 0$. Similar to proof of (ii), we also have $\phi(x^*,y^*)\geq0$. So we have $\phi(x^*,y^*) = 0$.
%
%As in (ii), we can show that $(x^*, y^*)$ is feasible for \reff{sipp:FsGH}
%and $H(x^*, y^*, z^*) =0$.

{Now suppose Algorithm \ref{alg:exchange:method} terminates within finitely many steps at $(\bar{x},\bar{y})\in \mathcal{X}^*$ with $\eps>0$. Then $(\bar{x}, \bar{y})$ must be a feasible solution to the approximate {SIPP} (\ref{pop:tilde(Pepsilon)}). Hence}
it is obvious that $(\bar{x}, \bar{y})$ is a global minimizer of (\ref{pop:tilde(Pepsilon)}).

\end{proof}

%By Theorem \ref{theorem:exchange:simple}, we know Algorithm \ref{alg:exchange:method} can always solve the approximate problem (\ref{pop:tilde(Pepsilon)}) globally in finite steps under some assumptions.

%To the end of this section, we would like to give a comment on Theorem \ref{theorem:exchange:simple}(ii) and (iii). For $\forall \epsilon>0$, However, if $\epsilon=0$ is not achievable, we might have big gap between $F_{\epsilon}^*$ and $F^*$, i.e., it is not always true $F_{\epsilon}^*\rightarrow F^*$ as $\epsilon\rightarrow 0$.

In Theorem \ref{theorem:exchange:simple},
we assumed that the subproblems $(P_k)$ and $(Q_i^k)$
can be solved globally by Lasserre relaxations.
This is a reasonably well assumption.
Please see the remarks A)-D) after Algorithm~\ref{alg:exchange:method}.
In the items (ii)-(iii),
the value function $v(x)$ {is} assumed to be continuous {at certain points}.
This can be satisfied under some conditions.
The {\it restricted inf-compactness} (RIC) is {such a condition}.
The value function $v(x)$ is said to have RIC at $x^*$ if $v(x^*)$ is finite and
there exist a compact set $\Omega$ and a positive number $\epsilon_0$,
such that for all $\| x - x^* \| < \eps_0 $ with $v(x) <v(x^*)+\epsilon_0$,
there exists $z\in S(x) \cap \Omega$.
For instance, if the set $Z$ is compact, or the lower level objective
$f(x^*,z)$ is weakly coercive in $z$  with respect to set $Z$, i.e.,
\[ \lim_{z\in Z, \|z\|\rightarrow \infty} f(x^*,z)=\infty, \]
then $v(x)$ has restricted inf-compactness at $x^*$; see, e.g.,
\cite[\S6.5.1]{Clarke}.
Note that the union $\cup_{k\geq 0} Z_k$ is contained in $Z$.
% \textcolor{blue}{ and $Z_k$ contains the global minimizers of function $f(x^k_i,z)$ over certain constraints which are subsets of $Z$.
So, if  $Z$ is compact
%or the function $f(x,z)$ is weakly coercive in $z$ uniformly in $x$,
then $\cup_{k\geq 0} Z_k$ is bounded.

\begin{prop}\label{valuefunction}
For the SBPP problem \reff{bilevel:pp},
assume the lower level  program is as in \reff{inopt:minfg}.
If the value function $v(x)$ has restricted inf-compactness at $x^*$,
then $v(x)$ is continuous at $x^*$.
\end{prop}
\begin{proof}
On one hand, since the lower level constraint is independent of $x$,
the value function $v(x)$ is always upper semicontinuous
\cite[Theorem 4.22 (1)]{Bank}. On the other hand,
since the restricted inf-compactness holds it follows from \cite[page 246]{Clarke}
(or see the proof of \cite[Theorem 3.9]{Guo-L-Y-Z})
that $v(x)$ is lower semicontinuous.
Therefore $v(x)$ is continuous at $x^*$.
\end{proof}

%  \textcolor{red}{}

\section{General Bilevel Polynomial Programs}
\label{sec:GBPP}

In this section, we study general bilevel polynomial programs as in \reff{bilevel:pp}.
For GBPPs, the feasible set $Z(x)$ of the {lower level program} \reff{df:P(x)}
varies as $x$ changes, i.e., the constraining polynomials $g_j(x,z)$
depends on $x$.

For each pair $(x,y)$ that is feasible for \reff{bilevel:pp},
$y$ is an optimizer for the {lower level  program \reff{df:P(x)} parameterized by $x$},
so $y$ must be a Fritz John point of \reff{df:P(x)}, i.e.,
there exists {$ (\mu_0,\mu_1, \ldots, \mu_{m_2}) \ne 0 $} satisfying
\[
\mu_0 \nabla_z f(x,y) - \sum_{j \in [m_2] } \mu_j \nabla_z g_j(x,y) = 0,
\quad \mu_j g_j(x,y) = 0 \, (j \in [m_2]).
\]
For convenience, we still use $K_{FJ}(x)$ to denote the set
of Fritz John points of \reff{df:P(x)} at $x$.
The set $K_{FJ}(x)$ consists of common zeros of some polynomials.
As in \reff{pro::opti:2}, choose the polynomials
$(f(z), g_1(z), \ldots, g_m(z))$ to be
$(f(x,z), g_1(x,z), \ldots, g_{m_2}(x,z))$,
whose coefficients depend on $x$.
Then, construct $\psi_1, \ldots, \psi_L$ in the same way as in
\reff{jac:poly:psi}. Each $\psi_j$ is also a polynomial in $(x,z)$.
Thus, every $(x,y)$ feasible in \reff{bilevel:pp} satisfies
$\psi_j(x,y)=0$, for all $j$. For convenience of notation, we still denote
the polynomial tuples $\xi, \psi$ as in \reff{sc3:xi:psi}.

We have seen that \reff{bilevel:pp} is equivalent to
the  {generalized} semi-infinite polynomial program
($H(x,y,z)$ is as in \reff{df:H(xyz)}):
\be \label{sip::GBPP}
\left\{
\baray{rl} F^* := \min\limits_{x\in \re^n, \, y\in \re^p}  & \, F(x,y) \\
\text{s.t.} & \, \psi(x,y) = 0, \, \xi(x,y) \geq 0, \\
& \, H(x,y,z) \geq 0, \, \forall~z\in Z(x).
\earay \right.
\ee
Note that the constraint $H(x,y,z) \geq 0$ in \reff{sip::GBPP}
is required for $z\in Z(x)$, which depends on $x$.
Algorithm \ref{alg:exchange:method} can also be applied to solve \reff{sip::GBPP}.
We first give an example for showing how it works.

\begin{exm}\label{exm:A7}
(\cite[Example 3.23]{BIEXM})
Consider the GBPP:
	\begin{equation}\label{exm:A7:p}
	\left\{
	\begin{aligned}
	\min\limits_{x,y\in [-1,1]}& \quad x^2\\
	\text{s.t.} & \quad 1+x-9x^2-y\leq 0, \\
	& \quad y\in \underset{z\in[-1,1]}{\tt argmin}	\{z \ \ \text{s.t.}~z^2(x-0.5)\leq 0\}.
	%\\
	%& \quad \quad \quad \quad \quad \text{s.t.}~z^2(x-0.5)\leq 0 .
	\end{aligned}
	\right.
	\end{equation}
By simple calculations, one can show that
	\[
	Z({x}) = \left\{
	\begin{aligned}
	& \{ 0 \},\quad \quad \quad \quad x\in (0.5,1],\\
	& [-1,1], \quad x\in [-1,0.5],
	\end{aligned}
	\right.
	\quad
	S(x) = \left\{
	\begin{aligned}
	& \{ 0 \},\quad \quad\quad x\in (0.5,1],\\
	& \{ -1 \}, \quad x\in [-1,0.5].
	\end{aligned}
	\right.
	\]
The set $\mathcal{U} = \{(x,y)\in [-1,1]^2: \, 1+x-9x^2-y\leq 0, y^2(x-0.5)\leq 0\}$.
The feasible set of \reff{exm:A7:p} is:
\[
\mathcal{F}:= \Big(
\left\{(x,0): x\in (0.5,1]\right\}\cup\left\{(x,-1): x\in [-1,0.5]\right\}
\Big)\cap \mathcal{U}.
\]
%
%It is
%the intersection of the red lines with the shaded area	
%in Figure (\ref{fig:exm:general:1}).
%
%The global optimal solution of bilevel program \reff{exm:A7:p} is  a  point in $\mathcal{F}$ such that $x^2$ is minimized. From the figure,
%
One can show that the global minimizer and the optimal values are
\[
(x^*,y^*) = \left(\frac{1-\sqrt{73}}{18},-1\right)
\approx (-0.4191, -1), \quad
F^* = \left(\frac{1-\sqrt{73}}{18}\right)^2 \approx 0.1757.
\]
By the Jacobian representation of Fritz John points, we get the polynomial
\[
\psi(x,y) = (x-0.5)y^2(y^2-1).
\]
We apply Algorithm \ref{alg:exchange:method} to solve \reff{exm:A7:p}.
The computational results are reported in Table~\ref{table:exmA7}.
\begin{table}[htb]
\caption{Results of Algorithm \ref{alg:exchange:method} for solving \reff{exm:A7:p}.}
\label{table:exmA7}
\centering
\begin{scriptsize}
\begin{tabular}{|c|c|c|c|c|c|c|c|} \hline
{\tt Iter} $k$   & $(x_i^k,y_i^k)$  & $z_{i,j}^k$ & $F_k^*$ &  $v_i^k$   \\ 	\hline
 0  & (0.0000, 1.0000)  & -1.0000 & 0.0000  & -2.0000  \\		\hline
 1  & (-0.4191, -1.0000) & -1.0000 & 0.1757 & -2.4e-11 \\ \hline
\end{tabular}
\end{scriptsize}
\end{table}
As one can see, Algorithm~\ref{alg:exchange:method} takes two iterations
to solve \reff{exm:A7:p} successfully.
\qed
\end{exm}

However, we would like to point out that
Algorithm~\ref{alg:exchange:method} might not solve GBPPs globally.
The following is such an example.

\begin{exm}\label{exm:A12}
	(\cite[Example 5.2]{BIEXM})
	Consider the GBPP:
	\be \label{exm:A12:p}
	\left\{ \baray{rl}
	\min\limits_{x \in \re, y\in \re} & \quad (x-3)^2+(y-2)^2\\
	\text{s.t.} & \quad 0 \leq x \leq 8, \quad y \in S(x),
	\earay	\right.
	\ee
	where $S(x)$ is the set of minimizers of the optimization problem
	\[
	\left\{ \baray{rl}
	\min\limits_{z \in \re} & (z-5)^2\\
	\text{s.t.}~& 0\leq z \leq 6, \quad -2x+z-1\leq 0, \\
	& x-2z+2\leq 0, \quad   x+2z-14\leq 0.
	\earay \right.
	\]
	It can be shown that
	\[
	S(x) = \left\{
	\begin{aligned}
	& \{ 1+2x \},& x\in [0,2],\\
	& \{ 5 \}, & x\in (2,4], \\
	& \{ 7-\frac{x}{2} \},& x\in (4,6],\\
	& {\emptyset},&  x\in (6,8].
	\end{aligned}
	\right.
	\]	
	The feasible set of \reff{exm:A12:p} is thus the set
	\[
	\mathcal{F}: =\{(x,y) \mid x\in [0,6], \,\, y\in S(x)\}.
	\]
	It consists of three connected line segments.
	%
	%It is the union of blue line segments in Figure \ref{fig:exm:fail:GBLPP}.
	%The problem \reff{exm:A12:p} is to find a point $(x^*,y^*)$
	%in $\mathcal{F}$ that is closest to $(3,2)$.
	%From Figure \ref{fig:exm:fail:GBLPP},
	%
	One can easily check that the global optimizer and the optimal values are
	\[
	(x^*,y^*) = (1,3), \quad \quad F^* = 5.
	\]
	The polynomial $\psi$ in the Jacobian representation is
	\[
	\psi(x,y) \, = \, (-2x+y-1)(x-2y+2)(x+2y-14)y(y-6)(y-5).
	\]
	We apply Algorithm \ref{alg:exchange:method} to solve \reff{exm:A12:p}.
	The computational results are reported in Table~\ref{table:exmA12p}.
	\begin{table}[htb]
		\caption{Results of Algorithm \ref{alg:exchange:method} for solving \reff{exm:A12:p}.}
		\label{table:exmA12p}
		\centering
		\begin{scriptsize}
			\begin{tabular}{|c|c|c|c|c|c|c|c|} \hline
{\tt Iter} $k$   & $(x_i^k,y_i^k)$  & $z_{i,j}^k$ & $F_k^*$ &  $v_i^k$   \\ 	\hline
				0  & (2.7996, 2.3998)  & 5.0021 & 0.2000  & -6.7611 \\	\hline
		%		1  & (2.9864, 5.0000)  & 5.0007 & 9.0101 &  -6.6e-6
			1  & (2.9972, 5.0000)  &  5.0021 &  9.0001 &   4.41e-6 \\ \hline
			\end{tabular}
		\end{scriptsize}
	\end{table}
For $\eps =10^{-6}$, Algorithm \ref{alg:exchange:method} stops at $k=1$,
	and {returns} the point $(2.9972, 5.0000)$,
	which is not a global minimizer.
	However, it is interesting to note that
	the computed solution $(2.9972, 5.0000) \approx (3,5)$,
	 a  local optimizer of problem \reff{exm:A12:p}.
%%%%%%%%%%%%
\iffalse

{\bf (
In Table 4, for $k=1$, the value $v_i^k = 1.10e-5$ is positive, not negative.
Something wrong here, or just numerical issues?
In the earlier version, $\eps = 10^{-4}$, its value is $-6.6e-6$, which is better.
Why is the new result worse than the old one? \textcolor{red}{Answer Prof. Nie's question: $v^k = 1.1e-5$ is reported by gloptipoly, which is not exact. Calculate by hand, $f(x^k,z^k)-f(x^k,y^k) = 0.0021^2 = 4.41e-6$.
	For $v^k$, theoretically, we always have negative value, until the first nonnegative value is found. Computationally, $v^k = f(x^k,z^k)-f(x^k,y^k)$. $(x^k,y^k)$ is found by solving $(P_k)$, then $z^k$ is found by solving $(Q_k)$. Numerical error always exists. So $v^k$ might be a positive value around 0. We have several examples. This result actually looks better, $x^k,y^k$ is closer to (3,5), $F_k$ looks more accurate. I put the code here, you might test on Windows system. Sometimes, windows might get better result. I do not have windows computer. Thanks.}
)}

\fi
%%%%%%%%%%%%%%%%%%%%%%%%%%%%%%
\qed

\end{exm}

Why does Algorithm \ref{alg:exchange:method} fail to
find a global minimizer in Example~\ref{exm:A12}?
By adding $z^0$ to the discrete subset $Z_1$,
the feasible set of $(P_1)$ becomes
\[
\{ x\in X,y\in Z(x)\} \cap \{  \psi(x,y) = 0 \} \cap \{  |y -5| \leq 0.0021\}.
\]
It does not include the unique global optimizer $(x^*,y^*)=(1,3)$.
In other words, the reason is that
$H(x^*,y^*,z^0) \geq 0$ fails to hold
and hence by adding $z^0$, the true optimal solution
$(x^*,y^*)$ is not in the feasible region of problem $(P_1)$.

From the above example, we observe that the difficulty
for solving GBPPs globally comes from
the dependence of the lower level feasible set on $x$.
For a global optimizer $(x^*,y^*)$,
it is possible that $H(x^*,y^*,z_{i,j}^k) \not\geq 0$
for some $z_{i,j}^k$ at some step,
i.e., $(x^*,y^*)$ may fail to satisfy the newly added constraint in $(P_{k+1})$:
$H(x,y,z_{i,j}^k)\geq 0$. In other words,
$(x^*,y^*)$ may not be feasible for the subproblem $(P_{k+1})$.
Let $\mathcal{X}_k$ be the feasible set of problem $(P_{k})$.
Since $Z_k\subseteq Z_{k+1}$, we have $\mathcal{X}_{k+1}\subseteq \mathcal{X}_k$
and $(x^*,y^*)$ is not feasible for $(P_{\ell})$, for all $\ell \geq k+1$.
In such case, Algorithm \ref{alg:exchange:method}
will fail to find a global optimizer.
However, this will not happen for SBPPs, since {$Z(x)\equiv Z$} for all $x$.
For all $z\in Z$, we have $H(x^*,y^*,z)\geq 0$, i.e., $(x^*,y^*)$
is feasible for all subproblems $(P_k)$.
This is why Algorithm \ref{alg:exchange:method}
has convergence to global optimal solutions
for solving SBPPs. However, under some further conditions,
Algorithm \ref{alg:exchange:method} can solve GBPPs globally.

\begin{theorem} \label{theorem:exchange:general}
For the general bilevel polynomial program as in \reff{bilevel:pp},
assume that the lower level program is as in \reff{df:P(x)}
and the minimum value $F^*$ is achievable at a point
$(\bar{x},\bar{y})$ such that
$H(\bar{x},\bar{y},z)\geq 0$ for all $z\in Z_k$ and for all $k$.
Suppose $(P_k)$ and $(Q_i^k)$
are solved globally by Lasserre relaxations.
%
%Let $\mathcal{X}^*$ be the set of global minimizers of \reff{bilevel:pp}.
%

\bit

\item [(i)] Assume $\eps=0$. If
 Algorithm \ref{alg:exchange:method} stops
for some $k < k_{\max}$, then each $(x^*,y^*) \in \mc{X}^*$
is a global minimizer of the GBPP problem \reff{bilevel:pp}.

\item [(ii)]
Assume $\eps=0$, $k_{\max}=\infty$, and the union
$\cup_{k\geq 0} Z_k$ is bounded.
Suppose Algorithm \ref{alg:exchange:method}
does not stop and each $S_k \ne \emptyset$ is finite.
Let $(x^*,y^*)$ be an arbitrary accumulation point of
the set $\cup_{k\geq 0} S_k$. If the value function $v(x)$,
defined as in \reff{def:v(x)}, is continuous at $x^*$,
then $(x^*, y^*)$ is a global minimizer of the GBPP problem \reff{bilevel:pp}.

\item [(iii)]
Assume $k_{\max}=\infty$,
the union $\cup_{k\geq 0} Z_k$ is bounded,
the set $\Xi = \{(x,y):\, \psi(x,y) =0, \xi(x,y) \geq 0\}$ is compact.
{Let $\Xi_1 = \{x :\, \exists y, (x,y)\in \Xi\}$,
the projection of $\Xi$ onto the $x$-space.
Suppose $v(x)$ is continuous on $\Xi_1$.}
Then, for all $\eps>0$, Algorithm \ref{alg:exchange:method}
must terminate within finitely many steps.

\eit	
\end{theorem}

\begin{proof} By the assumption, the point $(\bar{x},\bar{y})$
is feasible for the subproblem $(P_k)$, for all $k$.
Hence, we have $F_k^*\leq F^*$.
The rest of the proof is the same as the proof of Theorem
\ref{theorem:exchange:simple}.
\end{proof}

%%%%%%%%%%%%%%%%%%%%%%
\iffalse

\begin{remark}\label{rem4.4} \textcolor{red}{The assumptions in Theorem \ref{theorem:exchange:general} guarantee that $F_{k}^*\leq F^*$ for all $k\in \N$, which may not always hold. Example \ref{exm:A12} is a counter example.} However, we always have $F_0^*\leq F^*$, \textcolor{blue}{since for $k=0$, $Z_0=\emptyset$,  the feasible set of problem $(P_0)$ always contains the one for problem \reff{sip::GBPP}.
%defined in \reff{set:feasible:P} is feasible to problem $(P_0)$.
 Thus if Algorithm \ref{alg:exchange:method} stops with $k=0$, then each $(x^*,y^*) \in \mc{X}^*$ is a global minimizer of GBPP problem \reff{bilevel:pp} up to small numerical error;  see e.g.  Example \ref{Ex5.7}}.
\end{remark}

\fi
%%%%%%%%%%%%%%%%%%%%%%%%%%%%%%%%%%%%%%%

In the above theorem, the existence of the point
$(\bar{x},\bar{y})$ satisfying the requirement may be hard to check.
If $v(x)$ has restricted inf-compactness at $x^*$
and the Mangasarian-Fromovitz constraint qualification (MFCQ)
holds at all solutions of the lower level problem (\ref{df:P(x)}),
then the value function $v(x)$ is Lipschitz continuous at $x^*$;
see \cite[Corollary 1]{Clarke}.
Recently, it was shown in \cite[Corollary 4.8]{Guo-L-Y-Z}
that the MFCQ can be replaced by a weaker condition called
quasinormality in the above result.

\section{Numerical experiments}
\label{sc:numexp}

In this section, we present numerical experiments for solving BPPs.
In Algorithm~\ref{alg:exchange:method},
the polynomial optimization subproblems are solved by
Lasserre semidefinite relaxations,
implemented in software {\tt Gloptipoly~3} \cite{Gloptipoly}
and the SDP solver {\tt SeDuMi} \cite{sedumi}.
The computation is implemented with Matlab R2012a on a
MacBook Pro 64-bit OS X (10.9.5) system with 16GB memory and 2.3 GHz Intel Core i7 CPU.
In the algorithms, we set the parameters
$k_{\max}=  20$ and $\epsilon = 10^{-5}$.
%{\bf (
%Isn't $k_{\max}=  10$ too small? How about $20, 50$ or even $100$?
%What if changing $\epsilon = 10^{-4}$ to $\epsilon = 10^{-6}$?
%%\textcolor{red}{Answer: (1) $k_{\max}$ can not increase that much, we can choose 15 maybe. Two reasons: firstly, numerical error accumulates as discrete set $\{Z_{k}\}$ increases; secondly, we use Jacobian equations here, Sedumi can not solve very large SDP. For very small problems, for example $n=p=1$, we can choose $k_{\max} = 50$; for relatively larger problem, for example $n=5$, it will be very slow. So choose $k_{\max} = 10$ or 15 may be reasonable. (2) Choose $\epsilon  =10^{-6}$ might be very difficult to converge. Usually when $\epsilon$ reaches $10^{-4}$, it will have little update. This is because of numerical error for $z_k$, the added constraint polynomial will have relatively larger numerical error, and it is difficult to Gloptipoly to get better solution. For most of the problem, if it converges quickly, it will have very littler numerical error. }
%)}
In reporting computational results,
we use $(x^*,y^*)$ to denote the computed global optimizers,
$F^*$ to denote the {value of the outer objective function $F$ at $(x^*,y^*)$},
$v^*$ to denote $\inf_{z\in Z }~H(x^*,y^*,z)$,
{\tt Iter} to denote the total of number of iterations for convergence,
and {\tt Time} to denote the CPU time taken to solve the problem
(in seconds unless stated otherwise).
%
%Note that \textcolor{blue}
%{since the tolerance parameter $\epsilon = 10^{-6}$ is sufficiently small},
%
When $v^* \geq -\epsilon$,
the computed point $(x^*,y^*)$
is considered as a global minimizer of $(P)$,
up to the tolerance $\epsilon$. {Mathematically,
to solve BPPs exactly, we need to set $\epsilon =0$. However,
in computational practice, the round-off errors always exist,
so we choose $\epsilon >0$ to be a small number.}

\subsection{Examples of SBPPs}

 \begin{exm} \label{exm:appendix:13}
(\cite[Example 3.26]{BIEXM})
Consider the SBPP:
 \begin{equation*}
 \left\{
 \baray{rl}
 \min\limits_{x\in \RR^2,y\in \RR^3}  &   x_1y_1+x_2y_2^2+x_1x_2y^3_3\\
 \text{s.t.} &   x\in [-1,1]^2, \,\,  0.1-x_1^2\leq 0, \\
 &   1.5-y_1^2-y_2^2-y_3^2\leq 0, \\
 &   -2.5+y_1^2+y_2^2+y_3^2\leq 0, \\
 &    y\in S(x), \\
%
% \text{arg}\min\limits_{z\in [-1,1]^3}
% f(x,z):= x_1z_1^2+x_2z_2^2+(x_1-x_2)z_3^2\\
% & \quad x\in [-1,1]^2,\quad y \in [-1,1]^3.
%
 \earay
 \right.
 \end{equation*}
where $S(x)$ is the set of minimizers of
\[
\min\limits_{z\in [-1,1]^3} \quad
x_1z_1^2+x_2z_2^2+(x_1-x_2)z_3^2.
\]
It was shown in \cite[Example 3.26]{BIEXM} that the unique global optimal solution is
$$x^* = (-1,-1), \quad y^* = (1,\pm 1,-\sqrt{0.5}).$$
% CPUTIME: 14.8300
Algorithm~\ref{alg:exchange:method} terminates after one iteration.
It takes about 14.83 seconds. We get
\[
x^* \approx (-1,-1), \quad y^* \approx (1,\pm 1,-0.7071),
\]
\[
F^* \approx -2.3536, \quad v^* \approx -5.71 \times 10^{-9}.
\]
%%%%%%%%%%%%%%%%%%%%%%%%%%%%%%%%%%%%
\iffalse

 New added equations generated by the Jacobian representation are:
 \begin{equation*}
 \begin{aligned}
 \psi_1(x,y) &:= 2x_1y_1(y_1^2-1)(y_2^2-1)(y_3^2-1)=0,\\
 \psi_2(x,y) &: = 2x_2y_2(y_1^2-1)(y_2^2-1)(y_3^2-1)=0,\\
 \psi_3(x,y) &: = 2y_3(x_1-x_2)(y_1^2-1)(y_2^2-1)(y_3^2-1)=0,\\
 \psi_4(x,y) &: = 4x_2y_1y_2(y_2^2-1)(y_3^2-1)=0,\\
 \psi_5(x,y) &: = 4y_1y_3(x_1-x_2)(y_2^2-1)(y_3^2-1)=0,\\
 \psi_6(x,y) &: = 4x_1y_1y_2(y_1^2-1)(y_3^2-1)=0,\\
 \psi_7(x,y) &:= 4y_2y_3(x_1-x_2)(y_1^2-1)(y_3^2-1)=0,\\
  \psi_8(x,y) &: = 4x_1y_1y_3(y_1^2-1)(y_2^2-1)=0,\\
  \psi_9(x,y) &:= 4x_2y_2y_3(y_1^2-1)(y_2^2-1)=0,\\
 \psi_{10}(x,y) &: = 8y_1y_2y_3(x_1-x_2)(y_3^2-1)=0,\\
 \psi_{11}(x,y) &:= 8y_1y_2y_3x_2(1-y_2^2)=0,\\
\psi_{12}(x,y) &:= 8y_1y_2y_3x_1(y_1^2-1)=0.\\  		
\end{aligned}
\end{equation*}

\fi
%%%%%%%%%%%%%%%%%%%%%%%%%%%%%%%%%%%%%
\end{exm}

 \begin{exm} Consider the SBPP:
 	\begin{equation}\label{SBPP:exm:5.2}
 		\left\{
 		\baray{rl}
 		\min\limits_{x\in \RR^2,y\in \RR^3}  &   x_1y_1+x_2y_2+x_1x_2y_1y_2y_3\\
 		\text{s.t.} &   x\in [-1,1]^2, \,\,  y_1y_2-x_1^2\leq 0, \\
 		&    y\in S(x), \\
 		\earay
 		\right.
 	\end{equation}
 where $S(x)$ is the set of minimizers of
\[
\left\{ \baray{rl}
\min\limits_{z\in \RR^3} & x_1z_1^2+x_2^2z_2z_3-z_1z_3^2\\
\text{s.t.}~& 1\leq z_1^2+z_2^2+z_3^2\leq 2.
\earay \right.
\]
Algorithm~\ref{alg:exchange:method} terminates after one iteration.
It takes about 13.45 seconds. We get
\[
x^* \approx (-1,-1), \,\quad y^* \approx (1.1097,0.3143,-0.8184),
\]
\[
F^* \approx -1.7095, \, \quad
v^* \approx -1.19 \times 10^{-9}.
\]
%It is easy to see that set $\mathcal{U}$ is compact,
%and for each $x\in [-1,1]^2$,
%we have $S(x)\neq \emptyset$, since the objective function
%of lower level program is a polynomial and the constraint set is compact.
%The value function $v(x)$ of lower level program is continuous
%and the feasible set of problem \reff{SBPP:exm:5.2} is nonempty and compact.
By Theorem \ref{theorem:exchange:simple},
we know $(x^*,y^*)$ is a global optimizer,
up to a tolerance around $10^{-9}$.
\end{exm}

\begin{exm}\label{example:mitsos}
We consider some test problems from \cite{BIEXM}.
For convenience of display, we choose the problems
that have common constraints $x \in [-1,1]$
for the outer level program and $z \in [-1,1]$
for the inner level program.
When Algorithm~\ref{alg:exchange:method} is applied,
all these SBPPs are solved successfully.
The outer objective $F(x,y)$, the inner objective $f(x,z)$,
the global optimizers $(x^*,y^*)$,
the number of consumed iterations {\tt Iter}, the CPU time taken to solve the problem,
the optimal value $F^*$, and the value $v^*$ are reported in Table~\ref{SBPP:small}.
In all problems, {except Ex. 3.18 and Ex. 3.19,}
the optimal solutions we obtained coincide with those given in \cite{BIEXM}.
For Ex. 3.18, the global optimal solution for minimizing the upper level objective $-x^2+y^2$ subject to constraints
$x,y\in [-1,1]$ is $x^*=1,y^*=0$. It is easy to check that $y^*=0$ is the optimal solution for the lower level problem parameterized by $x^*=1$ and hence $x^*=1,y^*=0$ is  also the unique global minimizer for the SBPP in Ex. 3.18.
{For Ex.~3.19, as shown in \cite{BIEXM},}
the optimal solution must have $x^*\in (0,1)$.
For such $x^*$, $S(x^*)=\{\pm \sqrt{x^*}\}$. Plugging $y=\pm \sqrt{x}$ into the upper level objective we have $F(x,y)= \pm x\sqrt{x}+\sqrt{x}+\frac{x}{2}$. It is obvious that the minimum over $0<x<1$ should occur when $y=\sqrt{x}$. So minimizing $F(x,y)=  x\sqrt{x}-\sqrt{x}+\frac{x}{2}$ over $0<x<1$ gives
$x^*=(\frac{\sqrt{13}-1}{6})^2{\approx 0.1886}$,
$y^*=\frac{\sqrt{13}-1}{6}{\approx 0.4343}.$
\begin{table}[htb]
\caption{Results for some SBPP problems in \cite{BIEXM}.
They have the common constraints $x \in [-1,1]$ and $z \in [-1,1]$.
 }
\label{SBPP:small}
\centering
\begin{scriptsize}
\begin{tabular}{|c|l|c|c|c|c|c|} \hline
Problem & \qquad \qquad SBPP  & $(x^*,y^*)$ & {\tt Iter}
& {\tt Time}  & $F^*$ &  $v^*$   \\ \hline
%\cite[Example 3.12]{BIEXM}*
%$\baray{l} F = -x+xy+10y^2 \\ f = -xz^2+ \half z^4 \earay $
%& (0.0312,  -0.0000) &  10 & -0.0312 & -4.88e-4  \\ \hline
% \cite[Example 3.13]{BIEXM}*
%$\baray{l} F = x-y \\ f = \half xz^2 - zx^3 \earay $
%& (-0.0216, 1.0000)  & 10 & -1.0216 & -2.02e-5 \\  \hline
% \cite[Example 3.14]{BIEXM}
Ex. 3.14 & $\baray{l} F = (x-1/4)^2 + y^2  \\ f = z^3/3 - xz   \earay $
& (0.2500, 0.5000) &  2 & 0.49 & 0.2500 & -5.7e-10 \\ \hline
% \cite[Example 3.15]{BIEXM}
Ex. 3.15 & $\baray{l} F = x+y  \\ f = xz^2/2-z^3/3   \earay $
& (-1.0000, 1.0000) &   2 & 0.42 &  2.79e-8 & -4.22e-8 \\  \hline
% \cite[Example 3.16]{BIEXM}
Ex. 3.16 & $\baray{l} F = 2x+y  \\ f =  -xz^2/2- z^4/4   \earay $
&  (-0.5, -1), (-1, 0) &  2 & 0.47 & -2.0000 & -6.0e-10 \\  \hline
% \cite[Example 3.17]{BIEXM}
Ex. 3.17 & $\baray{l} F = (x+1/2)^2 + y^2/2   \\ f = xz^2/2+z^4/4   \earay $
& $(-0.2500, \pm 0.5000)$ &    4 & 1.12 & 0.1875 & -8.3e-11  \\ \hline
% \cite[Example 3.18]{BIEXM}
Ex. 3.18 & $\baray{l} F = -x^2 + y^2   \\ f =  xz^2- z^4/2  \earay $
&  (1.0000, 0.0000) &   2 & 0.44 & -1.0000  & -3.1e-13 \\  \hline
% \cite[Example 3.19]{BIEXM}
Ex. 3.19 & $\baray{l} F = xy - y + y^2/2   \\ f = -xz^2+z^4/2  \earay $
&  (0.1886, 0.4343)  &   2 &  0.41 & -0.2581 & -3.6e-12 \\  \hline
% \cite[Example 3.20]{BIEXM}
Ex. 3.20 & $\baray{l} F = (x-1/4)^2 + y^2    \\ f =  z^3/3-x^2z  \earay $
&  (0.5000, 0.5000) &  2  &  0.38 & 0.3125 &  -1.1e-10 \\  \hline
\end{tabular}
\end{scriptsize}
\end{table}
%
%The results are presented in Table \ref{SBPP:small}.
%We can see that most of the examples are solved globally by certifying $|v^{*}|<\epsilon$. All problems except  \cite[Example 3.12]{BIEXM} and \cite[Example 3.13]{BIEXM} converge to the required accuracy in $k_{\max}=10$ steps. Moreover each of the problems \cite[Examples 3.16and 3.26]{BIEXM} has two global minimizers and Algorithm \ref{alg:exchange:method} finds them accurately in a few steps.
% For \cite[Example 3.12]{BIEXM}, the exact global minimizer is $(x^*,y^*) = (0,0)$. Our computational result is quite close to this exact solution. But there are some numerical errors due to the fact that  the coefficents of the lower level subproblem $P(x^*)$ is very small. For \cite[Example 3.13]{BIEXM}, the exact global minimizer $(x^*,y^*) = (0,1)$ with the objective function value $F^* = -1$. Our computation result is quite close to the exact global minimizer but has some numerical errors. The difficulty may come from the fact that  at $x^*=0$ the lower level objective function vanishes, i.e., $f(x^*,z)\equiv0$ and hence the set of global minimizers is $S(0) = Z$, which has infinitely many global minimizers.
%
\end{exm}

 \begin{exm}Consider the SBPP:
 	\begin{equation}\label{exm:bigger:SBPP}
 		\left\{
 		\baray{rl}
 		\min\limits_{x\in \RR^4,y\in \RR^4}  &   x^2_1y_1+x_2y_2+x_3y_3^2+x_4y_4^2\\
 		\text{s.t.} &   \|x\|^2  \leq 1, \,\,  y_1y_2-x_1\leq 0,\\
 		&  y_3y_4-x_3^2\leq 0,\,\,  y\in S(x), \\
 		\earay
 		\right.
 	\end{equation}
 	where $S(x)$ is the set of minimizers of
 	\[
 	\left\{ \baray{rl}
 	\min\limits_{z\in \RR^4} & z_1^2-z_2(x_1+x_2)-(z_3+z_4)(x_3+x_4)\\
 	\text{s.t.}~& \|z\|^2  \leq 1,\,\,  z_2^2+z_3^2+z_4^2-z_1\leq 0.
 	\earay \right.
 	\]
We apply Algorithm \ref{alg:exchange:method} to solve \reff{exm:bigger:SBPP}.
The computational results are reported in Table~\ref{table:exm:bigger:SBPP}. As one can see, Algorithm \ref{alg:exchange:method} stops when $k=4$
and solves \reff{exm:bigger:SBPP} successfully.
It takes about 20 minutes to solve the problem.
%Clearly, the set
%$\mathcal{U}$ is compact. For each $x$ in the unit ball, $S(x)\neq \emptyset$,
%since the lower level program is defined as a polynomial over a compact set.
%The value function $v(x)$ is continuous.
%The feasible set of problem \reff{exm:bigger:SBPP} is nonempty and compact.
%
%By Theorem \ref{theorem:exchange:simple}(i), at Step 4,
%since $v^4$ is small enough, we have $(x^4,y^4)$
%is a certified global optimal solution of problem \reff{exm:bigger:SBPP}.
%
By Theorem \ref{theorem:exchange:simple},
%we know $(x^*,y^*)\approx (0.0000,-0.0000,-0.7071,
%-0.7071,0.6180,0.0000,-0.5559,-0.5559)$
we know the point $(x_i^k,y_i^k)$ obtained at $k=4$
is a global optimizer for
\reff{exm:bigger:SBPP}, up to a tolerance around $10^{-8}$.
%{At the iteration with $k=4$, the value $v_i^k$
%is almost zero,  so $(x_i^k,y_i^k)$ is a global optimizer of
%problem \reff{exm:bigger:SBPP}, up to a small numerical error.
%}

\begin{table}[htb]
	\caption{Results of Algorithm \ref{alg:exchange:method} for solving \reff{exm:bigger:SBPP}.}
	\label{table:exm:bigger:SBPP}
	\centering
	\begin{scriptsize}
		\begin{tabular}{|c|c|c|c|c|c|c|c|} \hline
		{\tt Iter} $k$   &   $(x_i^k,y_i^k)$  &   $F_k^*$ &  $v_i^k$   \\ 	\hline
			0  & (-0.0000,1.0000,-0.0000,0.0000,0.6180,-0.7862,~0.0000,~0.0000)  & -0.7862 & -1.6406 \\
			\hline
			1  &  (0.0000,-0.0000,0.0000,-1.0000,0.6180, -0.0000,0.0000,-0.7862) &     -0.6180  &  -0.3458 \\
			&  (0.0003,-0.0002,-0.9999,0.0000,0.6180,~0.0001,-0.7861,-0.0000) &   -0.6180  &  -0.3458 \\
			\hline
			2  & (0.0000,-0.0000,-0.8623,-0.5064,0.6180,-0.0000,-0.6403,-0.4561) &     -0.4589 &  -0.0211 \\
			\hline
			3  & (0.0000,-0.0000,-0.7098,-0.7042,0.6180,-0.0000,-0.5570,-0.5548) &     -0.4371 &  -6.37e-5\\ \hline
			4 & (0.0000,-0.0000,-0.7071,-0.7071,0.6180,0.0000,-0.5559,-0.5559)
	& -0.4370& -2.27e-8	\\	\hline
		\end{tabular}
	\end{scriptsize}
\end{table}
 \end{exm}
%{\bf (If we choose $k$ to bigger values like $4,5$,
%can we get $v_i^k > - 10^{-6}$? \textcolor{red}{Yes.})}

An interesting special case of SBPPs is that
the inner level  {program} has no constraints, i.e.,
$Z = \re^p$. In this case, the set $K_{FJ}(x)$ of Fritz John points
is just the set of critical points of the inner objective $f(x,z)$.
It is easy to see that the polynomial $\psi(x,y)$ is given as
\[
\psi(x,z)=\Big( \frac{\pt}{\pt z_1} f(x,z),
\ldots, \frac{\pt}{\pt z_p} f(x,z) \Big).
\]

\begin{exm}  \label{exm:sbpp:Z=R^p}
(SBPPs with $Z=\re^p$)
Consider random SBPPs with ball conditions on $x$
and no constraints on $z$:
\begin{equation}\label{exm:random:prob:uncons}
 \left\{
	\begin{aligned}
		F^*:=\min\limits_{x\in \re^n, \, y\in \RR^{p}}&~~ F(x,y)\\
		\text{s.t.} & ~~  \|x\|^2  \leq 1, \quad
    y \in \underset{z\in \RR^p}{\tt argmin} f(x,z),
	\end{aligned}
	\right.
\end{equation}	
where $F(x,y)$ and $f(x,z)$ are generated randomly as
%
%\begin{equation*}
%\begin{aligned}
%F(x,y) & := \eta_1^{T}[u]_{2d_1-1}+ ([u]^{d_1})^{T} M_1^{T} M_1[u]^{d_1},\\
%f(x,z) & := \eta_2^{T}\langle [x]_{2d_2-1}, [z]_{2d_2-1}\rangle+ \langle [x]^{d_2}, [z]^{d_2}\rangle^{T} M_2^{T} M_2\langle [x]^{d_2}, [z]^{d_2}\rangle,
%\end{aligned}
%\end{equation*}
%
\begin{equation*}
\begin{aligned}
F(x,y) & := a_1^{T}[u]_{2d_1-1}+ \| B_1[u]^{d_1} \|^2,\\
f(x,z) & := a_2^{T} [x]_{2d_2-1}+ a_3^T [z]_{2d_2-1} +
\left \| B_2 \bpm [x]^{d_2} \\ [z]^{d_2} \epm  \right\|^2.
\end{aligned}
\end{equation*}
In the above, $x=(x_1,\ldots,x_n),~y=(y_1,\ldots,y_p),~ z= (z_1,\ldots,z_p)$,
$u=(x,y)$ and $d_1, d_2\in \N$.
The symbol $[x]_d$ denotes the vector of monomials in $x$
and of degrees $\leq d$, while $[x]^d$ denotes the vector of monomials
in $x$ and of degrees equal to $d$.
The symbols $[y]_d, [y]^d, [u]^d$ are defined in the same way.

\begin{table}[htb]
\caption{Results for random SBPPs
	as in \reff{exm:random:prob:uncons} }.
\label{Table:random:SBPP:X:uncons}
\centering
	\begin{scriptsize}
		\begin{tabular}{|c|c|c|c|c|c|c|c|c|c|r|r|r|} \hline
			\multirow{2}{*}{$n$ }	& \multirow{2}{*}{$p$}  & \multirow{2}{*}{ $d_1$} &  \multirow{2}{*}{$d_2$} &  \multicolumn{3}{c|}{{\tt Iter} }  & \multicolumn{3}{c|}{{\tt Time} }   &  \multicolumn{3}{c|}{{\tt $v^*$}}   \\ \cline{5-13}	
			& & & & {\tt Min} & {\tt Avg} & {\tt Max} & {\tt Min}& {\tt Avg} & {\tt Max} & {\tt Min\quad} & {\tt Avg\quad ～ }& {\tt Max\quad }     \\ 	\hline
2 & 3 & 3  & 2 & 1 & 1.9 & 5 & 00:01 & 00:02 & 00:06 & -3.8e-6 & -2.9e-7 & -4.32e-8   \\ \hline
3 & 3 & 2  & 2 & 1& 1.6 & 2 & 00:04 & 00:07 & 00:09 & -4.0e-6 & -3.7e-7 &  -1.1e-10 \\ \hline
3 & 3 & 3  & 2 & 1 & 1.7 & 2& 00:04 & 00:07 & 00:10 & -2.0e-6  & -2.6e-7  &  -7.4e-11     \\ \hline
4 & 2 & 2  & 2 & 1 &  1.4 & 3 & 00:04 & 00:06 & 00:09 &  -3.0e-6& -2.4e-7 & -4.9e-12 \\ \hline
4 & 3 & 2  & 2 & 1 & 2.3 & 5 & 00:15 & 00:41 & 01:36 &  -5.3e-6 & -6.4e-7 &  -4.67e-9 \\ \hline
5 & 2 & 2  & 2 & 1 & 1.9 & 4 & 00:14 & 00:33 & 01:13 & -3.5e-6  & -8.1e-7 & -4.3e-11\\ \hline
5 & 3 & 2  & 2 & 1 &  1.8 & 3 & 06:30 & 10:04 & 11:56  & -1.1e-6 & -3.8e-7 & -1.9e-10 \\ \hline
6 & 2 & 2  & 2 & 1 & 2.0 & 4 & 04:02 & 09:56 & 17:39 & -6.2e-6 & -1.5e-6 & -5.57e-7 \\ \hline
\end{tabular}
	\end{scriptsize}
\end{table}

%{\bf (
%In Table~\ref{Table:random:SBPP:X:uncons}, I suggest to remove the case
%column {\tt Inst}, because it is always $20$.
%We can mention this in the example, but not in the table.
%Is this fine???  Moreover, for each case, it never take more than one hour.
%How about writing the time in the format like {\tt 14:02}
%instead of {\tt 0:14:02}? \textcolor{red}{Changed.}
%)}
%
%But for $n \leq 6$, can we solve the problem more accurately,
%so that $Avg(v^*) > - 10^{-6}$, by increasing $k_{max}$ or dropping $\eps$? \textcolor{red}{Answer: Random case, there are always examples can not converge in $k_{\max}$ steps, then $v^{*}$ is relatively large. Increase $k_{\max}$
%can get a little bit better error, but cputime increases.
%I update the results. $Ave(v^*)$ a little bit better.
%If we remove the large scale problem, we can remove Inst Column.}
%
We test the performance of Algorithm~\ref{alg:exchange:method}
for solving SBPPs in the form~\reff{exm:random:prob:uncons}.
The computational results are reported in Table~\ref{Table:random:SBPP:X:uncons}.
In the table, we randomly generated 20 instances for each case.
%{\tt Iter} denotes the number of iterations taken by
%Algorithm~\ref{alg:exchange:method}, {\tt Time} denotes the consumed time,
%	and {\tt $v^*$} denotes the values $v^*$. The consumed computational time is in the format {\tt mn:sc}, with mn and sc standing for minutes and seconds respectively. {\tt Min}, {\tt Avg} and {\tt Max} respectively stands for the minimum minimum, average and maximum of the quantities.
{\tt AvgIter} denotes the average number of iterations
taken by Algorithm~\ref{alg:exchange:method},
{\tt AvgTime} denotes the average of consumed time,
and {\tt Avg($v^*$)} denotes the average of the values $v^*$.
The consumed computational time is in the format {\tt mn:sc},
with mn and sc standing for minutes and seconds respectively.
As we can see, these SBPPs were solved successfully.
In Table~\ref{Table:random:SBPP:X:uncons}, the computational time
in the last two rows are much bigger than those in the previous rows.
This is because the newly added Jacobian equation
$\psi(x,y)=0$ has more polynomials and has higher degrees. Consequently,
in order to solve $(P_k)$ and $(Q_i^k)$ globally by Lasserre relaxations,
the relaxation orders need to be higher.
This makes the semidefinite relaxations more difficult to solve.

\end{exm}

\begin{exm} (Random SBPPs with ball conditions)
Consider the SBPP:
\be \label{exm:random:prob:2}
\left\{
\begin{aligned}
\min\limits_{x\in \re^n, \,y\in \re^p}&~~ F(x,y)\\
\text{s.t.} & ~~ \|x\|^2 \leq  1,
 \quad y \in  \underset{ \|z\|^2 \leq 1}{\tt argmin}
f(x,z).
%& \quad \quad \text{s.t.}~\|x\|^2+\|z\|^2 \leq 1 \\
%& \quad X = B_n(0,1),\quad
\end{aligned}
\right.
\ee
The outer and inner objectives $F(x,y)$, $f(x,z)$ are generated as
\[
F(x,y) = a^{T}[(x,y)]_{2d_1}, \quad
f(x,z)=  \bpm [x]_{d_2} \\ [z]_{d_2} \epm^T  B
\bpm [x]_{d_2} \\ [z]_{d_2} \epm.
\]
The entries of the vector $a$ and matrix $B$ are generated randomly
obeying Gaussian distributions. The symbols like
$[(x,y)]_{2d_1}$ are defined similarly as in Example~\ref{exm:sbpp:Z=R^p}.
We apply Algorithm~\ref{alg:exchange:method} to solve \reff{exm:random:prob:2}.
The computational results are reported in Table~\ref{Table:random:SBPP:X:exchange}.
The meanings of {\tt Inst}, {\tt AvgIter}, {\tt AvgTime},
and {\tt Avg($v^*$)} are same as in Example~\ref{exm:sbpp:Z=R^p}.
As we can see, the SBPPs as in \reff{exm:random:prob:2}
can be solved successfully by Algorithm~\ref{alg:exchange:method}.
% % % % % old one % % % % % % % % %
% % % % % % % % % % % % % % % % % % %
\iffalse

\begin{table}[htb]
\caption{Results for random SBPPs in
\reff{exm:random:prob:2}.}
\label{Table:random:SBPP:X:exchange}
\centering
\begin{scriptsize}
\begin{tabular}{|c|c|c|c|c|c|c|c|r|r|r|c|c|} \hline
$n$ & $p$ & $d_1$& $d_2$&    {\tt AvgIter} & {\tt AvgTime} & {\tt Avg($v^*$) }  \\  \hline
			3 & 2 & 2 & 2 &  2.6  & 00:03   &-1.4e-7       \\ \hline
			3 & 3 & 2 & 2 &  2.7  & 00:09 &   -6.5e-7   \\  \hline
			3 & 3 & 3 & 2 &   3.0  &   00:09  & -3.6e-7  \\ \hline
			4 & 2 & 2 & 2 &   3.5   & 00:20  & -5.0e-7    \\ \hline
			4 & 3 & 2 & 2 &  2.6 &   00:31 & -3.0e-7     \\ \hline
			5 & 2 & 2 & 2 &   3.7   & 00:43  & -1.7e-7    \\ \hline
			5 & 2 & 3 & 2 &   3.4  & 00:41 &   -5.4e-7    \\
			\hline
			6 & 2 & 2 & 2 &    2.6 & 09:17  & -5.7e-7     \\ \hline
			6 & 2 & 3 & 2 &  2.4  &  08:23  & -1.5e-7    \\ \hline
\end{tabular}
\end{scriptsize}
\end{table}

\fi
%%%%%%%%%%%%%%%%%%%%%%%%%%%%%%%%%%%%%%%%%%%%%%

\begin{table}[htb]
	\caption{Results for random SBPPs in
		\reff{exm:random:prob:2}. }
	\label{Table:random:SBPP:X:exchange}
	\centering
	\begin{scriptsize}
		\begin{tabular}{|c|c|c|c|c|c|c|c|c|c|r|r|r|} \hline
			\multirow{2}{*}{$n$ }	& \multirow{2}{*}{$p$}  & \multirow{2}{*}{ $d_1$} &  \multirow{2}{*}{$d_2$} &  \multicolumn{3}{c|}{{\tt Iter} }  & \multicolumn{3}{c|}{{\tt Time} }   &  \multicolumn{3}{c|}{{\tt $v^*$}}   \\ \cline{5-13}	
			& & & & {\tt Min} & {\tt Avg} & {\tt Max} & {\tt Min}& {\tt Avg} & {\tt Max} & {\tt Min\quad} & {\tt Avg\quad ～ }& {\tt Max\quad }     \\ 	\hline
			3 & 2 & 2 & 2 & 1 & 2.6 & 6 &  00:01 & 00:03 & 00:06 & -7.4e-7 & -1.4e-7 & 2.0e-9      \\ \hline
			3 & 3 & 2 & 2 & 1 & 2.7 & 6 & 00:03 & 00:09 & 00:21 &  -2.6e-6 & -6.5e-7 &-1.5e-9  \\  \hline
			3 & 3 & 3 & 2 &  1& 3.0  & 5 & 00:03 & 00:09 & 00:17 &  -2.9e-6 & -3.6e-7 & -1.1e-9 \\ \hline
			4 & 2 & 2 & 2 & 1 & 3.5 & 8 & 00:03 & 00:20 & 00:43 & -1.8e-6 & -5.0e-7 & 1.4e-9    \\ \hline
			4 & 3 & 2 & 2 & 1 & 2.6 & 5 & 00:12 & 00:31 & 01:01 & -2.9e-6 & -3.0e-7 & 1.8e-9   \\ \hline
			5 & 2 & 2 & 2 & 1 & 3.7 & 11  & 00:11 & 00:43 & 02:06   & -3.9e-6 & -1.7e-7 & -3.4e-9   \\ \hline
			5 & 2 & 3 & 2 & 1 & 3.4 & 10 & 00:10 & 00:41 & 02:15 & -3.6e-6 &  -5.4e-7 & -1.5e-9   \\
			 \hline
			6 & 2 & 2 & 2 &  1 & 2.6 & 6 & 03:21 & 09:17 &  22:41  & -4.3e-6 & -5.7e-7 & 5.8e-10    \\ \hline
			6 & 2 & 3 & 2 & 1 & 2.4 & 5 & 03:15 & 08:23 & 17:42 & -6.2e-7	& -1.5e-7	& 2.7e-10   \\ \hline
		\end{tabular}
	\end{scriptsize}
\end{table}
\end{exm}

\subsection{Examples of GBPPs}

%{\bf (
%In the examples here,  can we get global solutions for GBPP?
%If yes, explain reasons; if not, just honestly say NOT SURE.
%One referee complained about this.
%\textcolor{red}{Answer prof. Nie's question: We can not claim global optimal.
%Example 5.9, 7 small examples, only 4 find references to claim it is global optimal.
%For the last 3, the global optimal solution is unknown.
%So we might only claim, it converges.}
%
%\textcolor{magenta}{Jane and Li, how about Example 5.7-5.8?
%Can we claim global optimality for sure?
%By Thm 4.3(i), can we claim global optimality?
%As I pointed out earlier, Thm 4.3 may not be correct???
%Please double check!!!
%}
%
%\textcolor{blue}{Jiawang and Li, we have found optimal solutions for ALL problems. Our solutions for all problems except no. 3 coincide with those reported and analyzed in the reference. For no. 3, we have shown that we have found the true solution while the one given in the reference is wrong.}
%
%)}

\begin{exm}\label{Ex5.7}
%	(\textcolor{green}{where this example comes from? Reference? Analysis for optimal solutions?})
Consider the GBPP:
	\begin{equation} \label{exm:GBPP:1}
	\left\{
	\baray{rl}
	\min\limits_{x\in \RR^2,y\in \RR^3}  &    \frac{1}{2}x_1^2y_1+x_2y_2^2-(x_1+x_2^2)y_3\\
	\text{s.t.} &   x\in [-1,1]^2, \,\, x_1+x_2-x_1^2-y_1^2-y_2^2\geq 0,\\
	&  y\in S(x), \\
	\earay
	\right.
	\end{equation}
	where $S(x)$ is the set of minimizers of
	\[
	\left\{ \baray{rl}
	\min\limits_{z\in \RR^3} & x_2(z_1z_2z_3+z_2^2-z_3^3)\\
	\text{s.t.}~& x_1-z_1^2-z_2^2-z_3^2\geq 0,\,\,  1-2z_2z_3\geq 0.
	\earay \right.
	\]
We apply Algorithm \ref{alg:exchange:method} to solve \reff{exm:GBPP:1}. Algorithm~\ref{alg:exchange:method} terminates at the iteration $k=0$.
It takes about 10.18 seconds to solve the problem. We get
	\[
	x^* \approx (1,1), \quad y^* \approx (0,0,1),
	\quad F_0^* \approx -2, \quad v^* \approx -2.95 \times 10^{-8}.
	\]
%
%\textcolor{blue}{By Theorem \ref{theorem:exchange:general}(i) and Remark \ref{rem4.4}, we know $(x^*,y^*)$ is a global optimizer up to small numerical error. }	
%
Since $Z_0 = \emptyset$,
we have $F_0^* \leq F^*$ (the global minimum value). Moreover,
$(x^*, y^*)$ is feasible for \reff{exm:GBPP:1}, so $F(x^*, y^*) \geq F^*$.
Therefore, $F(x^*, y^*) = F^*$ and $(x^*,y^*)$ is a global optimizer,
up to a tolerance around $10^{-8}$.
\end{exm}

\begin{exm}
% 	(\textcolor{green}{where this example comes from? Reference? Analysis for optimal solutions?})
 	Consider the GBPP:
 	\begin{equation}\label{exm:bigger:GBPP}
 	\left\{
 	\baray{rl}
 	\min\limits_{x\in \RR^4,y\in \RR^4}  &   (x_1+x_2+x_3+x_4)(y_1+y_2+y_3+y_4)\\
 	\text{s.t.} &   \|x\|^2  \leq 1, \,\,  y_3^2-x_4\leq 0,\\
 	&\ y_2y_4-x_1\leq 0, \,\,  y\in S(x), \\
 	\earay
 	\right.
 	\end{equation}
 	where $S(x)$ is the set of minimizers of
 	\[
 	\left\{ \baray{rl}
 	\min\limits_{z\in \RR^4} & x_1z_1+x_2z_2+0.1z_3+0.5z_4-z_3z_4\\
 	\text{s.t.}~& z_1^2+2z_2^2+3z_3^2+4z_4^2\leq x_1^2+x_3^2+x_2+x_4,\\
 	& z_2z_3-z_1z_4\geq 0.
 	\earay \right.
 	\]
We apply Algorithm \ref{alg:exchange:method} to solve \reff{exm:bigger:GBPP}.
The computational results are reported in Table~\ref{table:exm:bigger:GBPP}.
Algorithm \ref{alg:exchange:method} stops with $k=1$.
It takes about 490.65 seconds to solve the problem.
We are not sure whether the point
$(x_i^k, y_i^k)$ computed at $k=1$ is a global optimizer or not.
\begin{table}[htb]
	\caption{Results of Algorithm \ref{alg:exchange:method} for solving \reff{exm:bigger:GBPP}.}
	\label{table:exm:bigger:GBPP}
	\centering
	\begin{scriptsize}
		\begin{tabular}{|c|c|c|c|c|c|c|c|} \hline
	{\tt Iter} $k$   &   $(x_i^k,y_i^k)$  &   $F_k^*$ &  $v_i^k$   \\ 	\hline
		0  & (0.5442,0.4682,0.4904,0.4942,-0.7792,-0.5034,-0.2871,-0.1855)  &  -3.5050 & -0.0391  \\
		\hline
		1  &  (0.5135,0.5050,0.4882,0.4929,-0.8346,-0.4104,-0.2106,-0.2887) &   -3.4880 & 3.29e-9  \\
			\hline
		\end{tabular}
	\end{scriptsize}
\end{table}
\end{exm}

\begin{table} [htb]
\caption{Results for some GBPPs} \label{GBPP:small}
	\centering \begin{scriptsize}
		\renewcommand{\arraystretch}{1.3}
		\begin{tabular}{|c|l|c|c|}
			\hline
			No. &
			 \quad \quad \quad \quad \quad\quad \quad \quad \quad\quad \quad \quad  \quad Small GBPPs &  \multicolumn{2}{|c| }{Results} \\
			\hline
			\multirow{6}{*}{ 1}	    &  \multirow{6}{*}{ $ \left\{
				\begin{array}{rl}  \min\limits_{x\in \RR,y\in \RR} & -x-y \\
				\text{s.t.} &  y \in S(x) := \underset{z\in Z(x)}{\tt argmin}~ z\\
				& Z(x):=\{z\in \RR| -x+z\geq 0,\, -z\geq 0\}.
				\end{array}   \right. $ } & $F^*$ & -2.78e-13 \\ \cline{3-4}
			&     & Iter & 1\\ \cline{3-4}
			&     & $x^*$& 3.82e-14 \\ \cline{3-4}
			&     & $y^*$& 2.40e-13 \\  \cline{3-4}
			&     & $v^*$ & -7.43e-13 \\ \cline{3-4}
			&    &  {\tt Time} &  0.19  \\
			\hline
			\multirow{6}{*}{ 2}	  &  \multirow{6}{*}{ $\left\{
				\begin{array}{rl} 	\min\limits_{x\in \RR,y \in \RR}& \quad (x-1)^2+y^2\\
				\text{s.t.} &  x\in [-3,2],\, y \in S(x) := \underset{z\in Z(x)}{\tt argmin}~z^3-3z\\
				& Z(x):=\{z\in \RR|z\geq x\}.
				\end{array}   \right. $} & $F^*$ & 0.9999 \\ \cline{3-4}
			&	 & Iter & 2\\ \cline{3-4}
			&	 & $x^*$& 0.9996\\ \cline{3-4}
			&	 & $y^*$& 1.0000  \\  \cline{3-4}
			&	 & $v^*$ & -4.24e-9\\ \cline{3-4}
			&  & {\tt Time}  &  0.57 \\
			\hline
			\multirow{6}{*}{ 3}	   &  \multirow{6}{*}{ $ \left\{
				\begin{array}{rl} \min\limits_{x\in \RR,y\in \RR}&  (x-0.6)^2+y^2\\
				\text{s.t.}  &  x,y \in [-1,1],\, y \in S(x) := \underset{z\in Z(x)}{\tt argmin}~f(x,z)=z^4+\frac{4}{30}(1-x)z^3\\
				&  +(0.16x-0.02x^2-0.4)z^2+(0.004x^3-0.036x^2+0.08x)z,\\
				& Z(x):=\{z\in \RR|0.01(1+x^2)\leq z^2,\, z\in [-1,1] \}.
				\end{array}   \right. $  } & $F^*$ & 0.1917 \\ \cline{3-4}
			&	 & Iter & 2\\ \cline{3-4}
			&	 & $x^*$& 0.6436\\ \cline{3-4}
			&	 & $y^*$& -0.4356  \\  \cline{3-4}
			&	 & $v^*$ & 2.18e-10\\ \cline{3-4}
			&	 & {\tt Time} &  0.52  \\
			\hline 			
			\multirow{6}{*}{4}	  &  \multirow{6}{*}{ $ \left\{
				\begin{array}{rl} \min\limits_{x\in \RR,y\in \RR^2}& \quad x^3y_1+y_2\\
				\text{s.t.} & x\in [0,1],\,y\in [-1,1]\times [0,100],\, y \in S(x) := \underset{z\in Z(x)}{\tt argmin} -z_2\\
				& Z(x):=\{z\in \RR^2|xz_1\leq 10,\, z_1^2+xz_2\leq 1, z\in [-1,1]\times [0,100] \}.
				\end{array}   \right. $  } & $F^*$ & 1 \\ \cline{3-4}
			&	 & Iter & 1\\ \cline{3-4}
			&	  & $x^*$& 1 \\ \cline{3-4}
			&	  & $y^*$& (0,1)  \\  \cline{3-4}
			&	  & $v^*$ & 3.45e-8\\ \cline{3-4}
			&	  & {\tt Time}  &  1.83 \\
			\hline
			\multirow{6}{*}{ 5}	 &        \multirow{6}{*}{ $\left\{
				\begin{array}{rl} 	\min\limits_{x\in \RR,y\in \RR^2}&  -x-3y_1+2y_2\\
				\text{s.t.} &  x\in [0,8],  y\in [0,4]\times [0,6],\, y\in S(x) = \underset{z\in Z(x)}{\tt argmin}
				-z_1\\
				& Z(x):=\left\{z\in \RR^2\left|\baray{c}-2x+z_1+4z_2\leq 16, 8x+3z_1-2z_2\leq 48\\
				2x-z_1+3z_2\geq 12,z\in [0,4]\times [0,6] \earay
				\right.  \right\}.
				\end{array}   \right. $ } & $F^*$ & -13 \\ \cline{3-4}
			&	  & Iter & 1\\ \cline{3-4}
			&	 		  & $x^*$& 5 \\ \cline{3-4}
			&  & $y^*$& (4,2)  \\  \cline{3-4}
			&			  & $v^*$ & 3.95e-6 \\ \cline{3-4}
			&			  & {\tt Time}  &  0.38 \\
			\hline
			\multirow{6}{*}{ 6}	   &  \multirow{6}{*}{ $ \left\{
				\begin{array}{rl} 	\min\limits_{x \in\re^2 ,y\in\re^2 }&\ -y_2\\
				\text{s.t.}&\  y_1y_2=0,x\geq 0,\, y\in S(x):= \underset{z\in Z(x)}{\tt argmin} ~z_1^2+(z_2+1)^2 \\
				&\ Z(x):=\left\{z\in \RR^2\left|\baray{c}(z_1-x_1)^2+(z_2-1-x_1)^2\leq 1,\\
				(z_1+x_2)^2+(z_2-1-x_2)^2\leq 1\earay
				\right. \right\}.  \end{array}
				\right. $ } & $F^*$ & -1 \\ \cline{3-4}
			&  & Iter & 2\\ \cline{3-4}
			&  & $x^*$& (0.71,0.71) \\ \cline{3-4}
			&  & $y^*$& (0,1) \\  \cline{3-4}
			&  & $v^*$ & -3.77e-10\\ \cline{3-4}
			&  & {\tt Time}  & 0.60 \\
			\hline	 		
			\multirow{6}{*}{7}	  	  &  \multirow{6}{*}{ $   \left\{
				\begin{array}{rl} 	\min\limits_{x\in \RR^2,y\in \RR^2}&  -x_1^2-3x_2-4y_1+y_2^2\\
				\text{s.t.} &  (x,y) \geq 0, {-x_1^2}-2x_2+4\geq 0,  y\in S(x):= \underset{z\in Z(x)}{\tt argmin}
				~z_1^2-5z_2\\
				& Z(x):=\left\{z\in \RR^2\left|\baray{c}~x_1^2-2x_1+x_2^2-2z_1+z_2 +3\geq 0,\\
				x_2+3z_1-4z_2-4\geq 0\earay
				\right\}. \right.
				\end{array}   \right. $  } & $F^*$ & -12.6787 \\ \cline{3-4}
			&	 & Iter & 2\\ \cline{3-4}
			&	  & $x^*$& (0,2) \\ \cline{3-4}
			&	  & $y^*$& (1.88,0.91) \\  \cline{3-4}
			&	  & $v^*$ &  2.40e-6 \\  \cline{3-4}
			&	  & {\tt Time}  &   10.52 \\ \hline	
		\end{tabular}
	\end{scriptsize}
\end{table}

\begin{exm}
%{\bf( For better display, I moved Table 10 to the front of Example 5.9)}
In this example we consider some GBPP examples given in the literature. The problems and the computational results are displayed in Table \ref{GBPP:small}.
% in \cite{allende2013solving,Dempe2012MP,BIEXM,YeSIAM2010}.
Problem  1 is    \cite[Example 3.1]{allende2013solving} and the optimal solution $(x^*,y^*)=(0,0)$ is reported.
Problem 2 is   \cite[Example 4.2]{YeSIAM2010} and the optimal solution
$(x^*,y^*)=(1,1)$ is reported.
Problem~3 is \cite[Example 3.22]{BIEXM}.
{As shown in \cite{BIEXM},}
the optimal solution should attain at a point satisfying $0< x<1$ and
$y=-0.5+0.1x$. For $(x,y)$ satisfying these conditions, the lower level constraint $0.01(1+x^2)-y^2\leq 0$ becomes inactive. Plugging $y=-0.5+0.1x$ into the upper level objective,  the bilevel program becomes finding the minimum of the convex  function
$(x-0.6)^2+(-0.5+0.1x)^2$. Hence the optimal solution is
$(x^*,y^*)=(\frac{65}{101}, \frac{44}{101})$.
Problem 4 can be found in \cite[Example 4.2]{BIEXM}
with the optimal solution $(x^*,y^*)=(1,0,1)$ reported.
Problem 5 can be found in  \cite[Example 5.1]{BIEXM} where
the optimal solution $(x^*,y^*)=(5,4,2)$ is derived.
Problem 6 is \cite[Example 3.1]{Dempe2012MP}.
{As shown in \cite{Dempe2012MP},}
the optimal solution is $(x^*,y^*)=(\sqrt{0.5}, \sqrt{0.5})$.
Problem 7 was originally given  in \cite[Example 3]{Bard}
and analyzed in \cite{allende2013solving}.
It was reported in \cite{allende2013solving} that the optimal solution is
$x^*=(0,2), y^* \approx (1.875, 0.9062)$.
In fact we can show that the optimal solution is
$x^*=(0,2), y^*=(\frac{15}{8}, \frac{29}{32})$ as follows.
Since the upper objective is separable in $x$ and $y$, it is easy to show that the optimal solution for
the problem
$$\min_{(x_1,x_2)\geq 0} -x_1^2-3x_2-4y_1+y_2^2
\quad \mbox{ s.t. } \quad -x_1^2-2x_2+4 \geq 0$$
with $y_1,y_2$ fixed
is $x^*_1=0, x_2^*=2$. Since $y^*=(\frac{15}{8}, \frac{29}{32})$ is the optimal solution to the lower level problem parameterized by $x^*=(0,2)$, we conclude that the optimal solution is $x^*=(0,2), y^*=(\frac{15}{8}, \frac{29}{32})$.
From Table \ref{GBPP:small},  we can see that Algorithm \ref{alg:exchange:method} {stops} in very few steps with
%\textcolor{red}{Theoretically, Algorithm \ref{alg:exchange:method} guarantees to solve GBPPs globally under a relatively stronger assumption, which might not be easy to check in practice. All these examples are chosen from references, which provides the best optimum that we can find.
%Although the condition that there exists at least one global minimizer for the bilevel program in each feasible region of subproblem $(P_k)$ is not easy to check,
%Using Algorithm \ref{alg:exchange:method}
 global optimal solutions for all  problems.
\end{exm}

\section{Conclusions and discussions}
\label{sc:condis}

This paper studies how to solve
both simple and general bilevel polynomial programs.
We reformulate them as equivalent semi-infinite polynomial programs,
using Fritz John conditions and Jacobian representations.
Then we apply the exchange technique and Lasserre type
semidefinite relaxations to solve them.
For solving SBPPs, we proposed Algorithm~\ref{alg:exchange:method}
and proved its {convergence}
to global optimal solutions.
For solving GBPPs, {Algorithm~\ref{alg:exchange:method}
can also be applied},
but its convergence to global optimizers is not guaranteed.
However, under some assumptions,
GBPPs can also be solved globally by Algorithm~\ref{alg:exchange:method}.
Extensive numerical experiments are provided
to demonstrate the efficiency of the proposed method.
{To see the advantages of our method,
we would like to make some comparisons with
two existing methods for solving bilevel polynomial programs.
The first one is the value function approximation approach
proposed by Jeyakumar, Lasserre, Li and Pham \cite{Lasbilevel2015};
the second one is the branch and bound approach proposed by
Mitsos, Lemonidis and Barton \cite{mitsos2008global}.
}

\subsection{Comparison with the value function approximation approach}

For solving SBPPs with convex lower level programs,
a semidefinite relaxation method was proposed in
\cite[\S3]{Lasbilevel2015},
under the assumption that the lower level programs satisfy
both the nondegeneracy condition and the Slater condition.
It uses multipliers, appearing in the Fritz John conditions,
as new variables in sum-of-squares type representations.
For SBPPs with nonconvex lower level programs,
it was proposed in \cite[\S4]{Lasbilevel2015}
to solve the following  $\epsilon$-approximation problem
(for a tolerance parameter $\eps>0$)
\be \label{bilevel:pp:eps}
(P^k_{\epsilon}): \left\{
\begin{aligned}
	F_{\epsilon}^{k} := \min\limits_{x\in \re^n,y\in \re^p}&\ F(x,y) \\
	\text{s.t.} \quad &\ G_i(x,y)\geq 0, \, i=1,\cdots,m_1, \\
	&\ g_j(y)\geq 0, j=1,\dots, m_2,\\
	& \  f(x,y)-J_{k}(x)\leq \epsilon.
\end{aligned}
\right.
\ee
In the above, $J_{k}(x)\in \RR_{2k}[x]$ is a $\frac{1}{k}$-solution
for approximating the nonsmooth value function
$v(x)$ \cite[Algorithm~4.5]{Lasbilevel2015}. For a given parameter $\epsilon >0$,
the method in \cite[\S4]{Lasbilevel2015}
finds the approximating polynomial $J_k(x)$ first,
and then solves $(P^k_{\epsilon})$ by Lasserre type semidefinite relaxations.
Theoretically, $\eps>0$ can be chosen as small as possible.
However, in computational practice, when $\eps>0$ is very small,
the degree $2k$ need to be chosen very high
and then it is hard to compute $J_k(x)$.
In the following, we give an example to compare our
Algorithm~\ref{alg:exchange:method} and the method in \cite[\S4]{Lasbilevel2015}.

%
% Usually, it needs to solve several semidefinite relaxations to find a good approximation of the value function, and it is also very difficult to decide if the approximation polynomial is good enough. As $k\rightarrow\infty$, and $ \epsilon \rightarrow 0 $, we have $F^k_{\epsilon}\rightarrow F^*$ asymptotically. Computationally, it is difficult to choose the parameter $ \epsilon $ and it might take several steps for the algorithm to converge. Even if the global optimum is found in finite steps, it is difficult to certify it. So in \cite{Lasbilevel2015}, there is no discussion about the finite convergence properties. For our algorithm, one big advantage is that if the finite convergence happens, we can easily certify it. For example,
%
%Find the best approximation polynomial to value function is not easy,
%its computationally expensive and not quite efficient.
%

\begin{exm} Consider the following SBPP:
	\begin{equation}\label{exm:comparsion:3}
	\left\{
	\begin{aligned}
	F^*:=\min\limits_{x\in \RR^2,y\in \RR^2}&\ y_1^3(x_1^2-3x_1x_2)-y_1^2y_2+y_2x_2^3\\	\quad \text{s.t.}&\ ~x\in [-1,1]^2,\ y_2+y_1(1-x_1^2)\geq 0,\ y\in S(x),
	\end{aligned}
	\right.
	\end{equation}
	where $S(x)$ is the solution set of the following optimization problem:
	\begin{equation*}
	v(x):=	\min\limits_{z\in \RR^2}~z_1z^2_2-z^3_2-z^2_1(x_2-x_1^2)\quad\text{s.t.}\quad z_1^2+z_2^2\leq 1.
	\end{equation*}	
The computational results of applying Algorithm~\ref{alg:exchange:method}
is shown in Table \ref{table:exm:comparsion:3}. It took only two steps
to solve the problem successfully. The set $\mathcal{U}$ is compact.
For each $x$, $S(x)\neq \emptyset$, since the lower level program is defined as a polynomial over a compact set. The value function $v(x)$ of lower level program is continuous. The feasible set of problem \reff{exm:comparsion:3}
is nonempty and compact.
%
%By Theorem \ref{theorem:exchange:simple}(i), at Step 1, since $v^1$ is small enough, we have $(x^1,y^1)=(0.5708,-1.0000,-0.1639,0.9865)$
%is a certified global optimal solution of problem \reff{exm:comparsion:3}.
%
{
At the iteration $k=1$, the value $v_i^k$ is almost zero,
so the point $(0.5708,-1.0000,-0.1639,0.9865)$
is a global optimizer of problem \reff{exm:comparsion:3},
up to a tolerance around $10^{-9}$.
}

\begin{table}[htb]
\caption{Computational results of Algorithm \ref{alg:exchange:method} for solving  \reff{exm:comparsion:3}.}
\label{table:exm:comparsion:3}
		\centering
\begin{scriptsize}
\begin{tabular}{|c|c|c|c|c|c|c|c|} \hline
{\tt Iter} $k$   &   $(x_i^k,y_i^k)$  & {$z_{i,j}^k$} & $F_k^*$ &  $v_i^k$   \\ 	\hline
0  & ( 1.0000,-1.0000,-1.0000,0.0000)  & (-0.1355,0.9908)& -4.0000 & -3.0689\\
				& (-1.0000,1.0000,-1.0000,0.0000)&(-0.2703,0.9628) & -4.0000
				& -1.1430 \\ 	\hline
1  & (0.5708,-1.0000,-0.1639,0.9865) & (-0.1638,0.9865) &-1.0219& -4.76e-9  \\ 	\hline
\end{tabular}
\end{scriptsize}
\end{table}

Next, we apply the method in \cite[\S4]{Lasbilevel2015}.
We use the software {\sf Yalmip} \cite{YALMIP}
to compute the approximating polynomial $J_k(x)\in\RR_{2k}[x]$,
as in \cite[Algorithm~4.5]{Lasbilevel2015}. After that,
we solve the problem {$(P_{\epsilon}^{k})$} by Lasserre type semidefinite relaxations,
for a parameter $\epsilon > 0$. Let $F_{\epsilon}^{k}$ denote
the optimal value of \reff{bilevel:pp:eps}.
The computational results are shown in Table \ref{table:exm:comparsion:3:Las}.
As $\epsilon$ is close to $0$, we can see that $F_{\epsilon}^{k}$
is close to the true optimal value $F^* \approx -1.0219$.
Since the method in \cite{Lasbilevel2015} depends on the choice of $\eps>0$,
we do not compare the computational time.
In applications, the optimal value $F^*$ is typically unknown.
An interesting question for research
is how to select a value of $\eps>0$
that guarantees $F_{\epsilon}$ is close enough to $F^*$.

\begin{table}[htb]
\caption{Computational results of the method in \cite[\S4]{Lasbilevel2015}.}
		\label{table:exm:comparsion:3:Las}
		\centering
		\begin{scriptsize}
			\begin{tabular}{|l|c|c|c|c|c|c|c|} \hline
			\quad 	$\epsilon$   &  $F_{\epsilon}^{2}$  & $F_{\epsilon}^{3}$   & $F_{\epsilon}^{4}$     \\ 	\hline
				1.0  & -3.4372 & -3.6423 & -3.6439 \\
				0.5  & -1.5506 &  -1.5909 & -1.5912
				\\
				0.25  &  -1.2718  & -1.2746 &  -1.2750   \\
				0.125  & -1.1746  & -1.1775 &  -1.1779 \\
				0.05  & -1.1193  & -1.1224 & -1.1228 \\
				0.01  &  -1.0897  & -1.0930 & -1.0934  \\
				0.005  & -1.0858 &  -1.0892 & -1.0897 \\
				0.001  &  -1.0827 &  -1.0862 & -1.0867 \\
				0.0001  & -1.0820   &  -1.0855 & -1.0860 \\
				\hline
			\end{tabular}
		\end{scriptsize}
	\end{table}

%
%{\bf(
%The lower bound $F_{\epsilon}^{k}$ is NOT monotonically increasing.
%Why does this happen? Is $P_{\epsilon}^{k}$ solved globally by GloptiPoly?
%What is the value of $F_{\epsilon}$? \textcolor{red}{Answer Prof. Nie's Question: (1) For fixed $k$, $J_k(x)$ is a given function. As $\epsilon$ decreases, the feasible region is smaller, so optimal $F_{\epsilon}^{k}$ is increasing. You can see each column is increasing; (2) For fixed $\epsilon$, $J_k(x)$ is increasing to approximate $J(x)$, intuitively, you can think about $f(x)-J_k(x)$ is smaller, feasible set $f(x)-J_k(x)\leq \epsilon$ is enlarged, so $F_{\epsilon}^{k}$ is decreasing. In the table, you can see each row is decreasing. Theoretically, I think this may be always true. In the paper \cite[Algorithm 4.5 Step 4]{Lasbilevel2015}, the optimal bound is chosen the minimal one over $[k]$. (3) The best $J_k(x)$ was found by Yalmip, $P_{\epsilon}$ is solved by GloptiPoly. (4) What's $F_{\epsilon}$, do you mean $k\rightarrow \infty$? In this case, $J_{\infty}(x) = J(x)$, so $F_{\epsilon}^{\infty}\leq F^*$, it is a lower bound. As $\epsilon \rightarrow 0$, have $F_{0}^{\infty} = F^{*}$. Computationally, can not get the value of $F_{\eps}$, since $J(x)$ is unknown.}
%)}
%

\end{exm}

\subsection{Comparison with the branch and bound approach}

Mitsos, Lemonidis and Barton \cite{mitsos2008global} proposed
a bounding algorithm for solving bilevel programs,
in combination with the exchange technique.
It works on finding a point that satisfies $\eps$-optimality
in the inner and outer programs. For the lower bounding algorithm,
a relaxed program {needs} to be solved globally.
The optional upper bounding problem is based on probing
the solution obtained by the lower bounding procedure.
The algorithm can be extended to use branching techniques.
For cleanness of the paper, we do not repeat the details here.
Interested readers are referred to \cite{mitsos2008global}.
We list some major differences between the method in our paper
and the one in \cite{mitsos2008global}.
\bit

\item The method in \cite{mitsos2008global} is based on
building a tree of nodes of subproblems,
obtained by partitioning box constraints for the variables $x,y$.
Our method does not need to build such a tree of nodes
and does not require box constraints for partitioning.

\item For each subproblem in the lower/upper bounding,
a nonlinear nonconvex optimization,
or a mixed integer nonlinear nonconvex optimization,
need to be solved globally or with $\eps$-optimality.
The software {\tt GMAS} \cite{GAMSManual} and {\tt BARON} \cite{BARON}
are applied to solve them.
In contrast, our method does not solve these nonlinear nonconvex subproblems
by  {\tt BARON} and {\tt GMAS}. Instead, we solve them globally
by Lasserre type semidefinite relaxations, which are convex programs
and can be solved efficiently by a standard
SDP package like {\tt SeDuMi}. In our computational experiments,
the subproblems are all solved globally
by {\tt GloptiPoly~3} \cite{Gloptipoly} and {\tt SeDuMi} \cite{sedumi}.

\eit

In \cite{mitsos2008global}, the branch and bound method was implemented
in {\tt C++}, and the subproblems were solved by {\tt BARON} and {\tt GMAS}.
In our paper, the method is implemented
in {\tt MATLAB}, the subproblems are solved by
{\tt GolptiPoly~3} and {\tt SeDuMi}.
Their approaches and implementations are very different.
It is hard to find a good way to compare them directly.
However, for BPPs, the subproblems in \cite{mitsos2008global}
and in our paper are all polynomial optimization problems.
To compare the two methods, it is reasonably well to compare
the number of subproblems that are needed to be solved,
although this may not be the best way.

We choose the seven SBPPs in Example~\ref{example:mitsos},
which were also in \cite{mitsos2008global}.
The numbers of subproblems are listed in
Table~\ref{table:exm:comparsion:B&b}.
In the table, {\sf B \& B (I)} is the branch and bound method
in \cite{mitsos2008global} without branching;
{\sf B \& B (II)} is the branch and bound method in \cite{mitsos2008global}
with branching; \#LBD is the number of lower bounding subproblems;
\#UBD is the number of upper bounding subproblems;
\#L-POP is the number of subproblems $(P_k)$ needs to be solved in
Algorithm \ref{alg:exchange:method}; \#U-POP is the number of subproblems
$(Q_i^k)$ needs to be solved in Algorithm~\ref{alg:exchange:method}.
The number of variables in lower bounding subproblems for
branch and bound methods (I/II) and subproblem $(P_k)$ for
Algorithm \ref{alg:exchange:method} are the same, all equal to $n+p$;
and the number of variables in upper bounding subproblems
for branch and bound methods (I/II) and subproblem $(Q_i^k)$
for Algorithm~\ref{alg:exchange:method} are the same,
all equal to $p$. For problem Ex.~3.16, since the subproblem
$(P_k)$ has two optimal solutions, so we need to
solve two subproblems {$(Q_i^k)$} to check if they are both
global optimal solutions.
From Table~\ref{table:exm:comparsion:B&b},
one can see that Algorithm~\ref{alg:exchange:method} has a
smaller number of subproblems that need to be solved.
If all the subproblems are solved by the same method,
Algorithm~\ref{alg:exchange:method} is expected to be more efficient.

%without branching;
%without branching;
%\#LBD(II) is the number of lower bounding subproblems with branching;
%\#UBD(II) is the number of upper bounding subproblems with branching;
%\#U-POP(Alg.~\ref{alg:exchange:method}) is the number of
%polynomial optimization problems solved by Algorithm \ref{alg:exchange:method}.
%Please note along with each lower bounding subproblem to be solved, another two subproblems \reff{alg:BB:sub2} and \reff{alg:BB:sub3} also need to be solved.
%
%
%{\bf (I do not understand the meaning of the preceding sentence.
%I guess few people know what is really going on in \cite{mitsos2008global}.
%Can we just list the total numbers of
%\#LBD(I), \#UBD(I), \#LBD(II), \#UBD(II)?
%Please update the results by the new counting!
%In our method, in each iteration, two POPs are solved.
%So, our total \# of POPs is two times of iterations, is this correct?
%The problems in Example~\ref{example:mitsos} are not labelled,
%but here we use the labelling no. 1-7, please be consistent!
%)}
%

\begin{table}[htb]
\caption{{ A comparison of the numbers of
polynomial optimization subproblems in \cite{mitsos2008global}
and in Algorithm \ref{alg:exchange:method}}.}
	\label{table:exm:comparsion:B&b}
	\centering
	\begin{scriptsize}
		\begin{tabular}{|c||c|c||c|c||c|c|c|} \hline
\multirow{2}{*}{ Problem}   	& \multicolumn{2}{c||}{{\sf B \& B (I)}} & \multicolumn{2}{c||}{{\sf B \& B (II)}} & \multicolumn{2}{c|}{Alg.~\ref{alg:exchange:method}} \\ \cline{2-7}		
 & \#LBD  &   \#UBD   &  \#LBD   &   \#UBD & \#L-POP & \#U-POP   \\ 	\hline
 Ex. 3.14 & 4 & 3 & 7 & 3 & 2 & 2\\  \hline
 Ex. 3.15 & 2 & 1 & 3 & 1 & 2 & 2 \\  \hline	
 Ex. 3.16 & 2 & 1 & 3 & 1 & 2 & 3 \\  \hline
 Ex. 3.17 & 19 & 18 & 37 & 18 & 4& 4\\  \hline	
 Ex. 3.18 & 2 & 2 & 3 & 2 & 2 & 2 \\  \hline	
 Ex. 3.19 & 13 & 12 & 27 & 14 & 2 & 2\\  \hline
 Ex. 3.20 & 4 & 3 & 5 & 3 & 2& 2 \\  \hline
		\end{tabular}
	\end{scriptsize}
\end{table}

\bigskip
\noindent
{\bf Acknowledgement}
Jiawang Nie was partially supported by the NSF grants
DMS-0844775 and DMS-1417985.
Jane Ye's work was supported  by the NSERC.
The authors would like to thank the anonymous reviewers
for the careful review and valuable comments that
helped improve the presentation of the manuscript.

\end{document}